\documentclass[11pt]{amsart}
\usepackage{tikz}
\usepackage{times}
\usepackage{color}
\usepackage{amsmath}
\usepackage{amssymb}
\usepackage{amsthm}
\usepackage{enumerate}
\usepackage{amsbsy}
\usepackage{amsfonts}
\newtheorem{theo}{Theorem}[section]

\newtheorem{prop}[theo]{Proposition}
\newtheorem{coro}[theo]{Corollary}
\newtheorem{lemm}[theo]{Lemma}
\newtheorem{rem}[theo]{Remark}

\newcommand{\al}{\alpha}
\newcommand{\be}{\beta}

\newcommand{\om}{\omega}
\newcommand{\Om}{\Omega}

\newcommand{\si}{\sigma}

\newcommand{\De}{\Delta}
\newcommand{\de}{\delta}

\newcommand{\pa}{\partial}

\newcommand{\R}{{\mathbb R}^n}

\newcommand{\na}{\nabla}

\begin{document}
\baselineskip=18pt

\title[Asymptotic properties]{Asymptotic properties  of the Stokes flow in an exterior domain with slowly decaying initial data and its application to the Navier-Stokes equations}

\
\author[T.Chang]{Tongkeun Chang}
\address{Department of Mathematics, Yonsei University \\
Seoul, 136-701, South Korea}
\email{chang7357@yonsei.ac.kr}

\author[B.J.Jin]{Bum Ja Jin}
\address{Department of Mathematics, Mokpo National University, Muan-gun 534-729,  South Korea }
\email{bumjajin@mnu.ac.kr}

\thanks{ Chang (RS-2023-00244630) and Jin (RS-2023-00280597)
are supported by the Basic Research Program
through the National Research Foundation of Korea
funded by Ministry of Science and ICT}

\begin{abstract}

In this paper, we study the decay rate of the Stokes flow  in an exterior domain  with  a slowly decaying  initial data  ${\bf u}_0(x)=O(|x|^{-\al}), 0<\al\leq n$. 
As an application we find  the unique strong solution of the Navier-Stokes equations corresponding to a   slowly decaying initial data.
We also derive the pointwise  decay estimate of the Navier-Stokes flow.  Our decay rates will be  optimal  compared with the decay rates of the heat flow. 


\end{abstract}

\maketitle

\section{Introduction}
\setcounter{equation}{0}

Let $\Omega$ be an exterior domain in ${\mathbb R}^n$, $n\geq 3$. 
In this paper we are interested in the asypmtotic behaviour of the Stokes and Navier-Stokes flow of  a slowly decaying initial data.

First, we consider  an  initial-boundary value problem for the Stokes equations    in $\Om\times (0,\infty)$:
\begin{align}
\label{e1-stokes}
\left\{\begin{array}{l}\vspace{2mm}
\partial_t{\bf u}-\Delta {\bf u} +\nabla p= {\rm div} \, {\mathcal F}\,\, \mbox{ in } \,\, \Omega\times (0,\infty),\\
\vspace{2mm}
\mbox{\rm div}\, {\bf u}=0 \,\, \mbox{ in } \Omega\times (0,\infty),\\
\vspace{2mm}
{\bf u}=0 \,\,\mbox{ on }\,\, {\partial\Omega}\times (0,\infty),\\
\vspace{2mm}
 \lim_{|x|\rightarrow \infty}{\bf u}(x,t)=   0\,\,\mbox{ for }\,\,t>0,\\
{\bf u}(x,0)={\bf u}_0(x)\,\, \mbox{ for }\,\, x\in  \Omega.
\end{array}\right.
\end{align}
Here, $ {\bf u}=(u_1,\cdots,u_n) $and $ p$ denote the unknown vector  field
and the unknown scalar function. respectively, while $ {\bf u}_0$ is a prescribed initial data.

For the simplicity we assume that  $B_{\frac12 }\subset \Omega^c\subset B_1,$ 
  where $B_r$ is the ball of radius $r$ centered at the origin.
We assume that the  initial data satisfies  the following two conditions: 
%
\begin{align}\label{H1}
\mbox{\rm div}\ {\bf u}_0=0 \mbox{ in }\Om \mbox{ in distributional sense,}
\end{align}
 \begin{equation}
 \label{H2}{\bf u}_0\cdot {\bf n}=0\mbox{ on }\partial \Om\mbox{ in trace sense}.
 \end{equation}
Here ${\bf n}$ is the unit outer normal vector to $\partial \Om$.
The above conditions are the compatibility conditions so that  the zero extension of ${\bf u}_0$  becomes divergence free  in ${\mathbb R}^n$.
%


Let ${\mathbb P}_q$ be the continuous projection operator from $L^q(\Omega)$ to $ J_q (\Omega) $(:=the completion of $C^\infty_{0,\sigma}(\Omega)$ in $ L^q(\Omega)$), and $A_q=-{\mathbb P}_q\Delta$ be the Stokes operator  
with dense domain ${\mathcal D}(A_q)=\{{\bf u}\in W^{2}_{q}(\Omega):  \ \mbox{\rm div}\ {\bf u}=0,\ {\bf u}|_{\partial \Omega}=0\}$. It is known that the Stokes operator $A_q$ generates a bounded analytic semigroup $e^{-tA_q}$.
From now on,  $A=A_q$ without confusion.
H. Iwashita \cite{iwashita} showed that
\begin{align*}
\|e^{-tA}f\|_{L^q(\Om)}&\leq ct^{-\frac{n}{2}(\frac{1}{r}-\frac{1}{q})}\|f\|_{L^r(\Om)}, \,\,1<r\leq q<\infty,\\
\|\nabla e^{-tA}f\|_{L^q(\Om)}&\leq ct^{-\frac{1}{2}-\frac{n}{2}(\frac{1}{r}-\frac{1}{q})}\|f\|_{L^r(\Om)},\,\, 1<r\leq q\leq n,
\end{align*} 
for any $f\in J_q(\Om)$ and $n\geq 3$.
P. Maremonti and V.A. Solonnikov\cite{maremonti-solonnikov} refined  the estimates  to 
\begin{align*}
\|\partial_t^k e^{-tA}f\|_{L^q(\Om)}&\leq ct^{-k-\frac{n}{2}(\frac{1}{r}-\frac{1}{q})}\|f\|_{L^r(\Om)},\  1\leq r< q\leq \infty,\mbox{ or }   1<r=q\leq  \infty, k=0,1,\\
\|\nabla e^{-tA}f\|_{L^q(\Om)}&\leq \left\{\begin{array}{l}
ct^{-\frac{1}{2}-\frac{n}{2}(\frac{1}{r}-\frac{1}{q})}\|f\|_{L^r(\Om)},\,\,  1<r\leq q\leq n\\
ct^{-\frac{n}{2r}}\|f\|_{L^r(\Om)},\ t\geq 1,\,\,  n<q<\infty
\end{array}\right.
\end{align*}
for any $f\in J_q(\Om)$ and $n\geq 3$. The result in  \cite{maremonti-solonnikov} includes  the   case $n=2$.
See also \cite{BM,BM2,dan,giga_sohr,kozono_ogawa,maremonti}. 

  According to the well known estimates in \cite{maremonti-solonnikov},  $\|e^{-tA}{\bf u}_0\|_{L^\infty(\Omega)}\leq ct^{-\frac{\al}{2}}\|{\bf u}_0\|_{L^{\frac{n}\al}(\Omega)}, \ \al\leq n.$ In this paper we consider a slowly decaying initial data  with  ${\bf u}_0(x)=O( |x|^{-\alpha}), \al\in (0, n]$. 
Observe that  ${\bf u}_0\notin L^{\frac{n}{\al}}(\Omega)$. Nonetheless, we will show that 
$\|e^{-tA}{\bf u}_0\|_{L^\infty(\Omega)}\leq ct^{-\frac{\al }{2}}\||x|^\al {\bf u}_0\|_{L^\infty(\Omega)},\ \al<n.$ 
The following is the precise statement of our first result.
\begin{theo}
\label{theorem_stokes1}%
Let $\Omega\subset{\mathbb R}^n$, $n\geq 3$ be an exterior domain of smooth boundary 
with $B_{\frac{1}{2}}\subseteq \Omega^c\subseteq B_1$. Let $ 0 <\al \leq n$ and $\frac{n}\al < q \leq \infty$. Assume that ${\bf u}_0$ satisfies the conditions \eqref{H1}-\eqref{H2} and 
 ${\bf u}_0=O(|x|^{-\alpha})$ for some $0<\alpha\leq  n$ with
\begin{align}
\label{H3}
\sup_{x\in \Omega}|x|^\alpha |{\bf u}_0(x)|:  = M_0 < \infty.
\end{align}

Then   it holds that 
\begin{align}
\label{weighted_stokes1}
\|e^{-tA}{\bf u}_0\|_{L^q(\Omega)}\leq  c _{\al, q} M_0 \left\{\begin{array}{l} \vspace{2mm}
 (1+ t)^{-\frac{ \alpha}{2} +\frac{n}{2q}},   \quad    0<\alpha<n,\\
   (1+t)^{-\frac{n}2 +\frac{n}{2q}} \ln (2 +t), \quad   \alpha=n.
\end{array}\right.
\end{align}



Moreover, it holds that  
\begin{align}
\label{stokes_infty1}
\begin{split}|
e^{-tA}{\bf u}_0(x)| & \leq \left\{\begin{array}{l} \vspace{2mm}
c_\al M_0(1+|x|+\sqrt{t})^{-\alpha}\quad \mbox{ if }\quad  0<\alpha\leq n-1
\\
 c_{\al,\delta }M_0(1+|x|+\sqrt{t})^{-\alpha+\delta} \,\, \mbox{ for any small } \delta>0, \,\mbox{ if } \,\,  n-1<\alpha\leq n.
 \end{array}\right.
\end{split}
\end{align}

\end{theo}

Our estimates in Thoerem \ref{theorem_stokes1}  is optimal in the sense that  the $L^q$ norm of 
the heat flow  corresponding to the intial data $(1+|y|)^{-\al}$, $0<\al\leq n$ behaves
like
$ (1+t)^{-\frac{\al}{2}}$ for $0<\al<n$ and $  (1+t)^{-\frac{\al}{2}} \ln{(2+t)}$ for $\al=n$(see Lemma \ref{lemma1}).

Slowly decaying data  for  the Stokes and the Navier-Stokes equations have been considered by T. Miyakawa\cite{miyakawa0} for the whole space  problem and  by  F. Crispo and P. Maremonti
\cite{crispo0} for the half space problem.
In both  papers, the optimal deay rates of the Stokes flow have been derived.
Their estimates comes from the direct estimte of the solution formula. However, there is no explicit solution formula for the exterior domain problem, and our method relies on the duality argument.
As far as we know, our result  is the first  showing the  optimal decay rates  for the  slowly decaying data in an  exterior domain problem.


According to the well known 
$L^q_tL^r_x$ estimate  in \cite{maremonti-solonnikov}(see  Proposition \ref{LqLr_stokes} of this paper)
\[
\|\nabla e^{-tA}{\bf u}_0\|_{L^q(\Omega)}\leq c_q \left\{\begin{array}{l}\vspace{2mm}
t^{-\frac{1}{2}}\|e^{-\frac{t}{2}A}{\bf u}_0\|_{L^q(\Om)}\quad t>0, \quad  1<q\leq n,\\
t^{-\frac{n}{2q}}\|e^{-\frac{t}{2}A}{\bf u}_0\|_{L^q(\Om)}\quad  t\geq 1,\quad   n\leq q<\infty.
\end{array}\right.
\]
Combining with the estimate of   Theorem \ref{theorem_stokes1} 
   we have the following estimates for the derivative of the Stokes flow.
\begin{coro}
\label{coro_stokes1}
Let $\Omega$, $\al, q, n$ and ${\bf u}_0 $  be the same as the ones appeared  in Theorem \ref{theorem1}. Then it holds
\begin{align}
\label{weighted_stokes2}
\|\nabla e^{-tA}{\bf u}_0\|_{L^q(\Omega)}\leq   c_{\al, q} M_0 \left\{\begin{array}{l} \vspace{2mm}
 t^{-\frac{1}{2}}(1+t)^{-\frac{\alpha }{2}+\frac{n}{2q}}, \quad t>0, \,\,\,  1 < \al < n, \,\,\,   \frac{n}{\al}<q\leq n \\
\vspace{2mm}
 t^{-\frac{1}{2}}(1+t)^{-\frac{ n }{2}+\frac{n}{2q}} \ln (2 +t),\,\,\,  t>0,\,\,\, \al =n,\,\,\,  1<q\leq n,\\
\vspace{2mm}
 t^{-\frac{\alpha}{2}}, \quad  t \geq 1,\,\,\,   0 < \al < n, \,\,\, \max\{\frac{n}{\al}, n\}<q<\infty,\\
 t^{-\frac{ n}{2}} \ln (1 +t), \quad  t\geq 1,  \,\,\,  \al =n, \,\,\,    n<q<\infty.
\end{array}\right.
\end{align}
\end{coro}

Second, we consider the initial and boundary value problem of the Navier-Stokes equations in $\Om\times (0,\infty)$:
%
\begin{align}
\label{e1}
\left\{\begin{array}{l}\vspace{2mm}
\partial_t{\bf u}-\Delta {\bf u} +\nabla p=- {\rm div} ({\bf u} \otimes {\bf u}) \,\, \mbox{ in } \,\, \Omega\times (0,\infty),\\
\vspace{2mm}
\mbox{\rm div}\, {\bf u}=0 \quad  \Omega\times (0,\infty),\\
\vspace{2mm}
{\bf u}=0\quad   \partial\Omega\times (0,\infty),\\
\vspace{2mm}
 \lim_{|x|\rightarrow \infty}{\bf u}(x,t)=0  \quad   0<t<\infty,\\
{\bf u}(x,0)={\bf u}_0(x) \quad  x\in  \Omega.
\end{array}\right.
\end{align}
By Duhamel's  principle, the solution of  \eqref{e1} is represented in the form of the following integral equation:
\begin{align}
\label{e1_integral}
{\bf u}=e^{-tA}{\bf u}_0-\int^t_0e^{-(t-\tau)A}{\mathbb P}\mbox{\rm div}({\bf u}\otimes {\bf u})(\tau) d\tau.
\end{align}
The $L^qL^r$ decay estimates of the Stokes semigroup could be applied to 
 the solvability of the  Navier-Stokes equaitons.

W. Borchers and T. Miyakawa \cite{BM,BM2} showed the solvability of the weak solutions for  an exterior domain problem for ${\bf u}_0\in J_2(\Om)$.
H. Iwashita \cite{iwashita} showed 
the global in time solvability  of $L^q$ strong solution of the Navier-Stokes  equations in an exterior domain  for the intial data ${\bf u}_0\in J_n(\Om)$. 
In those papers, the temporal decay of the solution have been investigated.
See  \cite{ galdi_maremonti,H2,maremonti1,maremonti,wiegner} for the analogous or improved results.
See also \cite{HM,kozono0} for the $L^1$ summability of the strong solutions. 

The estimates of Theorem \ref{theorem_stokes1} and Corollary \ref{coro_stokes1}  make us  have the unique solution of  \eqref{e1_integral} optimally decaying in time   for a   the slowly decaying data with  ${\bf u}_0=O( |x|^{-\alpha}), \al\in (0, n]$. The second result of our paper  reads as follows.
\begin{theo}
\label{theorem1}
Let $\Omega\subset{\mathbb R}^n$, $n\geq 3$ be an exterior domain of smooth boundary 
with $B_{\frac{1}{2}}\subseteq \Omega^c\subseteq B_1$. Let $ 1\leq \al \leq n$ and $\frac{n}\al < q <\frac{n}{\al-1}$. Assume that ${\bf u}_0$ satisfies the conditions \eqref{H1}-\eqref{H2} 
 ${\bf u}_0=O(|x|^{-\alpha})$ for some $0<\alpha\leq  n$. Let 
$\sup_{x\in \Omega}|x|^\alpha |{\bf u}_0(x)|:  = M_0.$

Then
there is small $M_0>0$ such that  
the equations \eqref{e1} has a unique gloabl strong solution    ${\bf u}\in C(0,\infty;L^q(\Om)\cap L^\infty(\Om))$ with 
$\nabla {\bf u}(t)\in C(0,\infty;L^n(\Om)),$
which has the properties
\begin{align}
\label{navier_weight_time}
\begin{split}
\|{\bf u}(t)\|_{L^q (\Om)}& \leq    c_{\al,q}M_0\left\{\begin{array}{l} \vspace{2mm}
 (1+ t)^{-\frac{ \alpha}{2} +\frac{n}{2q}}  ,  \quad 1\leq \alpha<n,\\
  (1 +t)^{-\frac{n}2 +\frac{n}{2q}} \ln (2 +t),   \quad  \alpha=n.
\end{array}\right.\\
\| {\bf u} (t)\|_{L^\infty (\Omega)} & \leq   c_{\al,q}M_0\left\{\begin{array}{l} \vspace{2mm}
  (1 + t)^{-\frac{ \alpha}{2} }  ,  \quad 1\leq \alpha<n,\\
 (1  +t)^{-\frac{n}2 } \ln (2 +t),   \quad  \alpha=n.
\end{array}\right.\\
\|  \nabla {\bf u}\|_{L^n (\Om))} 
& \leq  c_{\al, q}M_0 \left\{\begin{array}{l} \vspace{2mm}
  t^{-\frac12}  (1 +  t)^{-\frac{ \alpha}{2} +\frac{1}{2}}  ,  \quad 1\leq\alpha<n,\\
  t^{-\frac12} (1+ t)^{-\frac{n}2 +\frac{1}{2}} \ln (2 +t),   \quad  \alpha=n.
\end{array}\right.
\end{split}
\end{align}

Moreover, the solution ${\bf u}$  is  the strong solution of  \eqref{e1} with some associate pressure $p$.
 \end{theo}

 We also derived  the pointwise estimate of the  Navier-Stokes flow obtained in Theorem \ref{theorem1}:
\begin{coro}
\label{navier_wieght_space}
Let $({\bf u}, p)$ be the solution obtained in Theorem \ref{theorem1}. Then
\begin{align}
\label{navier_wieght_space}
\begin{split}
 |{\bf u}(x,t)| & \leq  \left\{\begin{array}{l} \vspace{2mm}
  c_{\al, q}M_0 (1+|x|+\sqrt{t})^{ -\alpha }, \quad   1\leq   \alpha \leq n -1,\\
c_{\al, q, \de}M_0  (1+|x|+\sqrt{t})^{-\alpha +\de } \mbox{ for any }\de>0,  \quad  n-1 <  \alpha \leq n .
\end{array}\right.
\end{split}
\end{align}

 \end{coro}

There are several literature on the weighted  estimate of the Navier-Stokes flow.
R. Farwig and H. sohr \cite{FS} studied the   space weighted estimate of solution such that  $|x|^\al  D^2_x {\bf u}, \,\,  |x|^\al D_t {\bf u} \in L^s (0, \infty; L^q (\Om))$ for $ 1 < q <\frac32, \,\, 1 < s<2$ and $0 \leq \frac3q +\frac2s -4 \leq \al < \min (\frac12, 3 - \frac3q)$.
B.J. Jin and H.-O. Bae \cite{bae-jin2}  showed that  there exists a  weak solution with  
\begin{align*}
\| |x| {\bf u}(t) \|_{L^2 (\Om)} \leq c_\de ( 1 + t) ^{\frac54 -\frac3{2r} + \de}\mbox{ for any }\de > 0,
\end{align*}
 if ${\bf u}_0 \in L^r(\Om) \cap L^2 (\Om)$ for some $1 < r< \frac65, \,\, |x| {\bf u}_0 \in L^{\frac65} (\Om)$ and $|x|^2 {\bf u}_0 \in L^2(\Om)$.
C. He and T. Miyakawa \cite{HM2} showed that if  $ {\bf u}_0 \in L^1 (\Om) \cap L^2_\si (\Om) \cap D^{1 -\frac1s,s}_p, \,\, n+1 =\frac2s +\frac{n}p$ and $|x|^\al {\bf u}_0 \in L^2(\Om)$ for some $1 < \al < \frac{n}2$, then the weak solution satisfies 
\begin{align*}
\| |x|^\be {\bf u}(t) \|_{L^2(\Om)} \leq c (1 +t)^{-\frac{n(\al -\be)}{4\al} } \quad 0 \leq \be \leq \al.
\end{align*}
When $n=3$, He and Xin \cite{he-xin} showed the existence of strong solution satifying $\||x|^\al {\bf u}(t)\|_{L^q (\Om)} \leq c$ for $\al = \frac37 -\frac3q$, $ 7 < q \leq \infty$.
In  \cite{bae-jin2}  it has been improved to  
\begin{align*}
\| |x|^2 {\bf u}(t) \|_{L^p (\Om)} \leq c_\de ( 1 + t) ^{1 -\frac{3}2( \frac1r -\frac1p) +\de}, 
\end{align*}
for all $\de >0$,  $1<p<\infty$ when ${\bf u}_0 \in L^r (\Om) \cap L^3 (\Om)$ for some $ 1 < r < \frac65$.
See  \cite{H,H2} for the improved results.
See also \cite{bae-jin} for the decay estimates of the 2D exterior domain problem.

\begin{rem}
1.  Local in time solvability for the  case $\al=0$  has been shown in \cite{AG}.
If we would consier the case $\al<1$ we could obtain local in time solvability and spatial asymptotic behavior of the solution. For the brevity  we considered only the case $\al\in [1,n]$ in Theorem \ref{theorem1}, which leads to obtaining  the global  in time solution and its asymptotic behavior.
(See \cite{maremonti2} for the half space problem, where the local in time solvability has been shown for $\al\in (\frac{1}{2},n)$ and global in time solvability has been shown for $\al\in [1,n)$.)

2. For the bounded nondecaying data see \cite{giga0,knightly0} for the whole space problem, \cite{maremonti2,solonnikov0} for the half space problem, \cite{AG} for the exterior domain problem. 
In particular, P. maremonti
  considered even a non convergent data at infinity  for  the whole space problem and half space problem 
and  the pointwise estimate for the Stokes flow has been derived for the data  with  ${\bf u}_0=O(1+|x|^\al), \al \in (0,1)$. 
Our technique used in this paper could be applied to the study of the exterior domain problem of the non decaying data. 
\end{rem}

This paper is organized as follows:  In section 2, we introduce necessary notations and function spaces, and  collect preparatory materials to prove main theorems.  In section 3 and section 4,  we prove Theorem \ref{theorem_stokes1} and Theorem \ref{theorem1}, respectively.

\section{Notations and   Preliminaries}
\setcounter{equation}{0}
\label{notation}

We introduce the notations used in this paper.  Let $D$ be a domain  in ${\mathbb R}^n$.   $C^\infty_0(D)$  denotes the set of    infinite times differentiable  functions compactly supported in $D$, and 
$C^\infty_{0,\sigma}(D)$ denotes  the set of infinitely differentiable  solenoidal vector fields compactly supported in $D$.
 For the nonnegative integer $k$  and $1\leq q\leq \infty$  $W^{k}_q(D)$  denotes the usual Sobolev space   and $W^{0}_{q}=L^q(D)$.
For $s>0$ and $1\leq p,q\leq \infty$  $W^{s}_{pq}(D)$ denotes  the usual Besov space.  It is known that  $W^{k}_{pp}(D)=W^k_p(D)$ and $C^\infty_0(D)$ is dense in $W^{s}_{pq}$ for $1\leq p,q<\infty$.

  For a Banach space $X$  and $1\leq p\leq \infty$ $L^p(0, T; X)$  denote the Banach space of  functions on  
the interval $(0,T)$ with values in $X$ 
with  the norm $\Big(\int^T_0\|\cdot (t)\|_{X}^p dt\Big)^{\frac{1}{p}}$. 
$C(0,T;X)$ denotes the set of continuous functions on the interval $(0,T)$ with values $X$ and $BC(0,T;X)$ denote the Banach space of bounded continuous functions on the interval $(0,T)$ with values in $X$ with the norm
$\sup_{0<t<T}\|\cdot (t)\|_X$.
 
In this paper  the symbol $c$ denotes   various generic constants and the symbol   $c_{*,\dots,*}$ denotes various  constants  depending on the parameters $*,\cdots, *$. We   will  use the generic constants $c$ when the parameters are not essential to our estimates. Otherwise,   we will use the constant $c_{*,\cdots,*}$.

Set 
 \[
J_q(\Om)=\mbox{ the  completion of }C^\infty_{0,\sigma}(\Om)\mbox{ in }L^q(\Om),\]
\[
G_q(\Om)=\{\nabla p\in L^q(\Om):  p\in L^q_{loc}(\Om)\}.
\]
It is well known that Helmholtz decomposition $L^q(\Om)=J_q(\Om) \oplus   G_q(\Om)$  holds  (see  \cite{Miyakawa}),
and the projection operator ${\mathbb P}: L^q(\Om)\rightarrow J_q(\Om)$ is continuous and bounded (see \cite{fujiwara-morimoto}).
 Let $A=-{\mathbb P}\Delta$ be the Stokes operator.
It is known that  $A$ generates a bounded analytic semigroup  $e^{-tA}$ (see \cite{borchers-sohr}).

In this paper, we use the following well known  $L^q-L^r$ decay estimates for the Stokes semigroup operator for $n\geq 3$. (See \cite{maremonti-solonnikov} for $n= 2$.)
\begin{prop} 
\label{LqLr_stokes}
Let $n\geq 3, $ $ 1\leq r< q\leq \infty$ or $1<r=q\leq \infty$.  
Let $f \in J_r(\Om)$. 

\begin{itemize}
\item[1)] 
\[
\|e^{-tA}f\|_{L^q(\Omega)}\leq ct^{-\frac{n}{2}(\frac{1}{r}-\frac{1}{q})}\|f\|_{L^r(\Omega)}, \quad  t>0.\]

\item[2)]
\begin{align*}
\|\nabla e^{-tA}f\|_{L^q(\Omega)}\leq \left\{\begin{array}{l} \vspace{2mm}
ct^{-\frac{1}{2}-\frac{n}{2}(\frac{1}{r}-\frac{1}{q})}\|f\|_{L^r(\Omega)},  \  t>0, \ \mbox{ if } 1 < q\leq n,   \\
ct^{-\frac{n}{2r}}\|f\|_{L^r(\Omega)}, \quad  t\geq 1,\  \mbox{ if }  \quad  n \leq q < \infty.
\end{array}\right.
\end{align*}

\item[3)]
\[\| A_q e^{-t A} f \|_{L^q(\Omega)} = 
\|\partial_te^{-tA}f\|_{L^q(\Omega)}\leq ct^{-1-\frac{n}{2}(\frac{1}{r}-\frac{1}{q})}\|f\|_{L^r(\Omega)}, \quad  t>0.\]
\end{itemize}

\end{prop}


For ${ \varphi}_{0}\in C^\infty_{0, \si}(\Omega)$ let   $\varphi=e^{-tA}\varphi_0$ . Then there is  $\pi$  satisfying  
\begin{align}
\label{e3}
\left\{\begin{array}{l}\vspace{2mm}
\partial_t{\varphi}-\Delta {\varphi}+\nabla \pi={\bf 0}\quad \mbox{ in } \quad \Omega\times (0,\infty),\\
\vspace{2mm}
\mbox{\rm div}{\varphi}=0 \quad \mbox{ in } \quad \Omega\times (0,\infty),\\
\vspace{2mm}
{\varphi}={\bf 0}\quad \mbox{ on }\quad {\partial\Omega}\times (0,\infty),\\
\vspace{2mm}
 \lim_{|x|\rightarrow \infty}{\varphi}(x,t)= 0\quad \mbox{ for }\quad t>0,\\
{\varphi}(x,0)={\varphi}_0\quad \mbox{ for }\quad x\in  \Omega,
\end{array}\right.
\end{align}

The following higher norm estimates are well known.
\begin{theo}[\cite{galdi, kozono}]
\label{stokes_higher}
Let $n\geq 3$.
 and let  
$1\leq q<r\leq \infty$ or  $1<q=r\leq \infty$.
Suppose that $(\varphi, \pi)$ be the solution of \eqref{e3}. Then, 
\begin{align*}
\|\nabla^2 \varphi\|_{L^q(\Omega)}+\|\nabla \pi\|_{L^q(\Omega)} & \leq c\|A_q \varphi \|_{L^q(\Omega)}, \quad  1<q<\frac{n}{2},\\
\|\nabla^2\varphi\|_{L^q(\Omega)}+\|\nabla \pi \|_{L^q(\Omega)}& \leq c\|A_q\varphi\|_{L^q(\Omega)}+\|\nabla \varphi \|_{L^s(\Omega)}, \quad  \frac{n}{2}\leq q\leq s <\infty.
\end{align*}
\end{theo}


Let  $T(\varphi,\pi)=\frac{1}{2}( \nabla \varphi+\nabla^\perp  \varphi)+ {\mathbb I} \pi$, where ${\mathbb I}$ is identity matrix in ${\mathbb R}^n$. 
The following estimate will be used in this paper.
\begin{lemm}
\label{lemma4} 
 Let  $n\geq 3$,  and let  
$1\leq q<r\leq \infty$ or  $1<q=r\leq \infty$.
Suppose that $(\varphi, \pi)$ be the solution of \eqref{e3}   with $\int_{\Omega_R} \pi dx=0$ for some $R>0$ with $\Om^c\subset B_R$.  Here $\Om_R=\Om\cap B_R$.

Then
  it holds that  
\[
\|T({\bf \varphi}(t),\pi(t))\|_{L^r(\partial\Omega)}\leq c \left\{\begin{array}{l} \vspace{2mm}
t^{-1-\frac{n}{2q}+\frac{n}{2r}}\|{\bf \varphi}_0\|_{L^q(\Omega)}, \quad  0 < t,\quad  1< r\leq \frac{n}{2},\\
\vspace{2mm}
t^{-\frac{n}{2q}}\|{\bf \varphi}_0\|_{L^q(\Omega)}, \quad  1 \leq t,  \quad  \frac{n}{2}\leq r.
\end{array}\right.
\]

\end{lemm}

\begin{proof}
According to Theorem 6.6.1 in \cite{BL} and the relation of embedding, we have 
\begin{align*}
\|T({\bf \varphi},\pi)\|_{L^r(\partial \Omega)} 
&\leq c \|T({\bf \varphi},\pi)\|_{W^{\frac1r}_{r,1}(\Omega_R )}\leq c \|T({\bf \varphi},\pi)\|_{W^{1,r}( \Omega_R)}\\
&\leq  c \big(\|\nabla^2 {\bf \varphi}\|_{L^r (\Om)}+\|\nabla \pi\|_{L^r( \Omega_R)}+ \|\nabla {\bf \varphi}\|_{L^r(\Omega_R)}+\| \pi\|_{L^r( \Omega_R)} \big).
\end{align*}

Since $\int_{\Omega_R}\pi dx=0$, by Poincare's inequality
\[
\| \pi\|_{L^r( \Omega_R)}\leq c_R\|\nabla \pi\|_{L^r(\Omega_R)}.
\]
Hence, we have 
\[
\|T({\bf \varphi},\pi)\|_{L^r(\partial \Omega)}\leq  
 c\Big( \|\nabla^2 {\bf \varphi}\|_{L^r(\Omega)}+\|\nabla \pi\|_{L^r( \Omega)}+ \|\nabla {\bf \varphi}\|_{L^s(\Omega_R)}\Big), \quad 1<r \leq s< \infty.
\]
%
Take $s=\frac{nr}{n-r}$ if $r<\frac{n}{2}$ and $s\geq \max\{n,r\}$ if $r\geq \frac{n}{2}$.
Applying   Theorem \ref{stokes_higher}  and  Proposition \ref{LqLr_stokes} to the right hand sides of the above estimates we obtain 
 the estimates  of  Lemma \ref{lemma4}.

\end{proof}

\begin{rem}

Set $k=-\frac{1}{|\Omega_R|}\int_{\Omega_R}\tilde{\pi} dx$ and ${\pi} =\tilde{\pi }+k$.  Then $\int_{\Omega_R}{\pi}  dx=0$.
If $({\varphi},\tilde{\pi })$ satisfies  \eqref{e3}  in the exterior domain $\Omega \subset {\mathbb R}^n$, then so does $({\varphi},{\pi })$.
\end{rem}

The fundamental solutions of heat equation is expressed by 
\begin{align*}
N(x) = \left\{\begin{array}{ll}\vspace{2mm} 
-\frac{1}{(n-2)\om_n} \frac{1}{|x|^{n-2}},& n \geq 3\\
\frac1{2\pi} \ln |x|, & n =2
\end{array}\right.
\end{align*}
and  the fundamental solutions to the Laplace equation is expressed by
\begin{align*}
\Gamma(x,t) (:=\Gamma_t)= \left\{\begin{array}{ll} \vspace{2mm}
\frac{1}{(4\pi t)^{\frac{n}2}} e^{ -\frac{|x|^2}{4t} }, &   t > 0,\\
0, & t < 0,
\end{array}
\right.
\end{align*}
where $\om_n$ is the surface measure of the unit sphere in $\R$.
Then the 
%
 fundamental solutions of the Stokes equations  is 
   represented by
 \[{\bf G}^i:=\nabla\times \nabla\times \omega^i=\Gamma(x,t){\mathbf e}^i-\nabla \partial_{x_i}N*\Gamma_t,\
 Q^i=-\partial_{x_i}N(x) \delta(t),\]
 which satisfies 
 \begin{align*}
  \partial_t{\mathbf G}^i-\Delta {\mathbf G}^i+\nabla Q^i=\delta(x)\delta(t){\mathbf e}^i,  \quad 
  \mbox{\rm div }{\mathbf V}^i=0.
  \end{align*}
 where  $\omega^i =N*\Gamma_t {\mathbf e}^i =\int_{{\mathbb R}^n}N(x-y)\Gamma(y,t)dy{\mathbf e}^i$
(see \cite{lady} for the details).

The solution of \eqref{e1-stokes} is  represented by the  following integral equations. 
(See  section 3.2 of \cite{bae-jin} and references therein.) 
\begin{lemm}
\label{integral_stokes0}
Let $\zeta$ be a smooth function with $\zeta=0$ in some neighborhood of $\partial \Om$.
Then  the following representation holds for the solution ${\bf u}$ of \eqref{e1-stokes}: 
\begin{align}
\label{integral_stokes}
\begin{split}
{ \bf u}_i(x,t)\zeta(x)& =\int_\Omega {\bf u}_{0,i}(y)\zeta(y) \Gamma(x-y,t)dy\\
& \quad -\int_\Omega {\bf u}_0(y) \cdot [(\nabla\zeta(y))\partial_{y_i}\omega^i(x-y,t) 
 - (\nabla\zeta(y))\times \nabla\times \omega^i(x-y,t)]dy\\
&\quad +\int_\Omega {\bf u}(y,t)\cdot [(\nabla\zeta(y))\partial_{y_i}N(x-y) -(\nabla\zeta(y))\times \nabla\times (N(x-y){\mathbf e}^i)]dy\\
&\quad +\sum_{j=1}^2\int^{t}_0\int_\Omega  {\bf u}(y,\tau)\cdot R^i_j(x,y,t-\tau) dyd\tau\\
& \quad - \int_0^t \int_\Om {\mathcal F}(y,\tau):\nabla  \Big(\zeta (y) {\bf G}^i (x-y, t-\tau)   \Big) dy d\tau\\
& \quad- \int_0^t \int_\Om {\mathcal F}(y,\tau):\nabla \Big( \na \zeta (y) \times [\na \times \om^i (x-y, t-\tau)] \Big)   dy d\tau,
\end{split}
\end{align}
where 
\begin{align*}
R^i_1(x,y,t-\tau) & =  
- 2\sum_{k=1}^n\nabla \partial_{y_k}\zeta(y)[\nabla_y\times \partial_{y_k}\omega^i(x-y,t-\tau)]-\nabla \Delta \zeta(y)\times  [\nabla_y\times \omega^i(x-y,t-\tau)],\\ 
R^2_i(x,y,t-\tau) & =-2(\nabla\zeta(y)\cdot \nabla_y){\mathbf G}^i(x-y,t-\tau)-\Delta\zeta(y){\mathbf G}^i(x-y,t-\tau).
\end{align*}
\end{lemm}


The following estimates can be  found in many literature, which can be obtained by decomposing the domain of integrations by $\{y:|y|<\frac{|x|}{2}\}, \{y: |x-y|<\frac{|x|}{2}\}, \{y:|t|\geq \frac{|x|}{2}, |x-y|\geq \frac{|x|}{2}\}$:
\begin{equation}
\label{equat1}
\big|\int_{{\mathbb R}^n}\partial_{x_j}^m\partial_{x_i}N(x-y)\partial_{x_k}^l\Gamma(y,t)dy\big|\leq c(|x|+\sqrt{t})^{-n+1+m+l}\quad  l, \,\, m=0,1,\cdots.
\end{equation}

The estimate  \eqref{equat1} equivalently  leads to the following one:
\begin{lemm}
\label{lemma6_1}
\begin{align}
|D_x^k\omega(x-y,t-\tau)|\leq c(|x-y|+\sqrt{t-\tau})^{-n+2-k},k=1,2,\cdots.
\end{align}
\end{lemm}

The folllowing estimate can be obtained by the  straightforward computations, and might be found in literatures(we  give its proof in  Appendix \ref{appendix.lemma1} for the clarity).
\begin{lemm}
\label{lemma1}
Let $v_0$ satisfy  
\[
\label{1.2}
|  {v}_0(x)|\leq   M_0 (1+|x|)^{-\alpha},\ x \in \R.
\] and let 
${ V}$ be defined by  
\begin{equation*}
{V} (x,t)=\int_{{\mathbb R}^n}\Gamma(x-y,t)v_0(y) dy.
\end{equation*}
 Then  
it holds that 
\begin{align}\label{lemma11219-1}
|{ V}(x,t)|\leq cM_0 
   (|x|+\sqrt{t}+1)^{-\alpha}  \ln^{\de_{\al n}}{(2+|x|)} , \quad  0 < \alpha \leq n.
\end{align}
  This also  leads  to the estimate 
\begin{align}
\label{stokes_U}
\|{ V}(t)\|_{  L^q ({\mathbb R}^n)}\leq  cM_0  
  (t+1)^{-\frac{\al}2+\frac{n}{2q}}  \ln^{\de_{\al n}}{ ( 2+t) },\quad 0 < \al \leq  n,  \quad \frac{n}\al < q \leq \infty.
\end{align}
 Here $\delta_{\al n}$ denotes the Kronecker delta function: $\delta_{\al n}=1$ if $\al =n$ and $\delta_{\al n}=0$ of $\al \neq n$.

\end{lemm}

The following estimate will be used  several times in this paper, whose proof is straightforward and we omit its proof.
\begin{lemm}
\label{1220-2}
Observe that 
 \begin{align}
 \int^{t}_0   (1 +\tau)^{- a}\ln^{\delta_{\al n}} {(2+\tau)}   d\tau\leq 
 \left\{\begin{array}{l} \vspace{2mm}
 c\mbox{ if }a>1\mbox{ and }\al\leq n,\\
 \vspace{2mm}
  c \ln{(2+ t)} \quad \mbox{if} \quad a =1\mbox{ and }\al<n,\\
  \vspace{2mm}
 c \ln^2{(2+ t)} \quad \mbox{if} \quad a =1\mbox{ and }\al=n,\\
 \vspace{2mm}
 c t^{1-a}\mbox{ if }a<1\mbox{ and }\al<n,\\
   c_\de t^{1-a+\de}\mbox{ for any small }\de>0, \mbox{ if }a<1\mbox{ and }\al=n.
  \end{array}\right.
 \end{align}
 \end{lemm}

\section{Proof of Theorem \ref{theorem_stokes1}}
\label{decomposition}
\setcounter{equation}{0}

In this section we show that  the estimate   in Theorem \ref{theorem_stokes1} holds.  
Define ${\bf u}=e^{-tA}{\bf u}_0. $ Since $\partial_t {\bf u}=-A_q{\bf u}
\in 
C(0,\infty;J_q(\Om))$,  there is $ p$ with $\nabla p \in  C(0,\infty; G_q(\Om))$ satisfying that   $A_q{\bf u}=-\Delta {\bf u}+\nabla p$. This implies that   $({\bf u}, p)$ satisfies  \eqref{e1-stokes} with ${\mathcal F}=0$.

{\bf Step 1: Temporal estimates in $L^q(\Om)$,  $\frac{n}{\alpha}<q\leq \infty$.}

In this step, we will derive the temporal estimate \eqref{weighted_stokes1} in Theorem \ref{theorem_stokes1}.
 Note that  ${\bf u}_0\in L^q(\Om)$.  By 
 Proposition \ref{LqLr_stokes}, 
\begin{equation}
\label{Lqstokes}
\|{\bf u} (t)\|_{L^q(\Omega)}\leq c\|{\bf u}_0\|_{L^q(\Omega)}\leq cM_0, \quad \frac{n}\al <q\leq \infty.
\end{equation}
Hence, we have only to study the case for  large $t$. 

From now on we let $t>2$.

{$\bullet$ \bf Estimate   of $\|{\bf u}(t)\|_{L^q(\Om)}$ for   $\frac{n}{\al}<q<\infty$.}


Let $\tilde{\bf u}_0$ be the zero extension of ${\bf u}_0$ to the whole space ${\mathbb R}^n$.  Note that Hypotheses \eqref{H1} and  \eqref{H2} imply  ${\rm div}  \,\tilde {\bf u}_0 =0$ in $\R$ and  the Hypothesis \eqref{H3} implies           \begin{equation}
\label{1.2}
| \tilde {\bf u}_0(x)|\leq   M_0 (1+|x|)^{-\alpha},\ x \in \R.
\end{equation}
Set
\begin{equation}
\label{1.3}{\bf U} (x,t)=\int_{{\mathbb R}^n}\Gamma(x-y,t)\tilde{\bf u}_0(y) dy.
\end{equation} 
and 
%
 ${\bf v}={\bf u}-{\bf U}$.
 Then, $({\bf v}, p)$ satisfies the equations
\begin{align}
\label{e_stokes}
\left\{\begin{array}{l}\vspace{2mm}
\partial_t{\bf v}-\Delta{\bf v}+\nabla p={\bf 0}\quad \Omega\times (0,\infty),\\
\vspace{2mm}
\mbox{\rm div}\,{\bf v}=0\quad \Omega\times (0,\infty),\\
\vspace{2mm}
{\bf v}|_{\partial \Omega}=-{\bf U}|_{\partial \Omega}\quad  t>0,\\
\vspace{2mm}
{\bf v}(x,0)={\bf 0}\quad \Omega,\\
\lim_{|x|\rightarrow \infty}{\bf v}(x,t)=0\quad  t>0.
\end{array}\right.
\end{align}

Observe that  ${\bf U}$ satisfies the estimate \eqref{stokes_U} in  Lemma \ref{lemma1}.  We will  show the same estimate  holds for ${\bf v}$ for any $t>2$:
\begin{align}\label{230504-6}
\|{\bf v}(t) \|_{ L^q (\Om)} 
\leq   cM_0  
  t^{-\frac{ \alpha}{2}+\frac{n}{2q}}  \ln^{\de_{\al n}} { (2 +t) } ,  \quad 0<\alpha\leq n.
\end{align}
Our estimate will be done via duality argument.

From now on, let $t>2$.
For ${ \varphi}_{0}\in C^\infty_{0, \si}(\Omega)$ let  $(\varphi,\pi)$ be the solution of the Stokes equations \eqref{e3}. 
Taking inner product $  \varphi(t-\tau)$ to $\eqref{e_stokes}_1$ and integrating by parts over  $ (x,\tau)\in \Omega \times (0,t)$ we obtain
\begin{align}\label{230504-1}
\begin{split}
\int_\Omega{\bf v}(x,t)\cdot \varphi_0(x) dx& =\int^{t-1}_0\int_{\partial\Omega}{\bf  U}(x,\tau)\cdot  T (\varphi(x, t-\tau),\pi(x, t-\tau))\nu dS d\tau\\
& \quad + \int_\Om {\bf U} (x, t-1) \cdot \varphi(x,1)dx -  \int_\Om {\bf U}(x,t) \cdot  \varphi_0 (x) dx\\
& = I_1 + I_2+ I_3.
\end{split}
\end{align}
Recall   \eqref{stokes_U} in Lemma \ref{lemma1}:
  $\|{\bf  U}(\tau)\|_{L^q(\pa \Om) } 
   \leq c M_0 (1 + \tau)^{-\frac{\al}2 +\frac{n}{2q} }\ln^{\delta_{\al n}} {(2+\tau)} $ for $ 0< \al \leq n,$ and $ \frac{n}{\al} <  q \leq \infty.$

     Take $r$ satisfying $\max\{q',\frac{n}{2}\}\leq r$.  Then,  from  Lemma \ref{lemma1} for $q=\infty$ and Lemma \ref{lemma4}, we have 
\begin{align*}
\begin{split}
I_1 & \leq c \int^{t-1}_0 \|{\bf  U}(\tau)\|_{L^\infty(\pa \Om) } \| T (\varphi(x, t-\tau),\pi(x, t-\tau)\|_{L^1 (\pa \Om)}  d\tau\\
& \leq c \int^{t-1}_0 \|{\bf  U}(\tau)\|_{L^\infty(\pa \Om) } \| T (\varphi(x, t-\tau),\pi(x, t-\tau)\|_{L^r (\pa \Om)}  d\tau\\
 &\leq c M_0  \|\varphi_0\|_{L^{q'}(\Omega)}  \int^{t-1}_0 (1 +\tau)^{- \frac{\al}2 } \ln^{\delta_{\al n}} {(2+\tau)}   ( t -\tau)^{-\frac{n}{2q'} }  d\tau.
 \end{split}
\end{align*}
Here,  
  \begin{align*}
  \begin{split}
& \int^{t-1}_0 (1 +\tau)^{- \frac{\al}2} \ln^{\delta_{\al n}} {(2+\tau)}( t -\tau)^{-\frac{n}{2q'}}  d\tau\\
  &  \leq    c t^{-\frac{n}{2q'}} \int^{\frac{t}{2}}_0   (1 +\tau)^{- \frac{\al}2} \ln^{\delta_{\al n}} {(2+\tau)}   d\tau +  c  t^{- \frac{ \al}2}\ln^{\delta_{\al n}} {(2+t)}  \int^{t-1}_{\frac{t}{2}} ( t -\tau)^{-\frac{n}{2q'}} d\tau  \\
  & : = I_{11} + I_{12}.
  \end{split}
 \end{align*}
Using  the estimate of Lemma \ref{1220-2} for   $\int^{\frac{t}{2}}_0   (1 +\tau)^{- \frac{\al}2} \ln^{\delta_{\al n}} {(2+\tau)}   d\tau$ and $\int^{t-1}_{\frac{t}{2}} ( t -\tau)^{-\frac{n}{2q'}} d\tau $
 we have
\begin{align*}
\begin{split}
I_1  \leq  c M_0  \|\varphi_0\|_{L^{q'}(\Omega)} 
  t^{-\frac{ \alpha}{2}+\frac{n}{2q}} \ln^{\delta_{\al n}} {(2+t)},  \quad 0<\alpha\leq n
 \end{split}
\end{align*}
 for $ \frac{n}\al < q < \infty$ 

Since ${\bf U}$ satisfies the estimate  in Lemma \ref{lemma1} and $\varphi$  satisfies  the estimate  (1) in Proposition \ref{LqLr_stokes} 
we have
\begin{align*}
\begin{split}
I_2 
 &  \leq \| {\bf U}(t-1)\|_{L^q (\Om)} \| \varphi(1)\|_{L^{q'} (\Om)}\\
&  \leq 
cM_0 
  \| \varphi_0\|_{L^{q'} (\Om)}   t^{-\frac{\al}2 +\frac{n}{2q}}\ln^{\delta_{\al n}} {(2+t)}  ,  \quad   0 <\al \leq n
\end{split}
\end{align*}
for $ \frac{n}\al < q < \infty$.
By the similar reasoning we have 
\begin{align*}
\begin{split}
I_3 \leq \| {\bf U}(t)\|_{L^q (\Om)} \| \varphi_0\|_{L^{q'} (\Om)}
 \leq   
 cM_0  
  \| \varphi_0\|_{L^{q'} (\Om)}  t^{-\frac{\al}2 +\frac{n}{2q}} \ln^{\delta_{\al n}} {(2+t)}  ,  \quad   0 < \al  \leq n. 
\end{split}
\end{align*}
Therefore we obtain the estimates \eqref{230504-6}  
 for $ \frac{n}\al < q<\infty$, again  this leads to the estimate 
 \eqref{weighted_stokes1} in Theorem \ref{theorem_stokes1} for $\frac{n}{\al}<q<\infty$.  

{$\bullet$ \bf Estimate of $\|{\bf u}(t)\|_{L^\infty (\Om)}   $.}

Observe that ${\bf u}(t)=e^{-\frac{t}{2}A }e^{-\frac{t}{2} A }{\bf u}_0.$
Fix  some $r$ with $\frac{n}{\al}<r<\infty$. Then $
e^{-\frac{t}{2}A }{\bf u}_0\in L^r(\Om)$  with 
\[
\|e^{-\frac{t}{2}A }{\bf u}_0\|_{L^r(\Om)}\leq  c M_0 
  t^{-\frac{\al}2 +\frac{n}{2r}}\ln^{\delta_{\al n}} {(2+t)} ,   \quad   0 < \al \leq n.
\]
By the well known $L^\infty-L^r$ decay rate estimates in Proposition \ref{LqLr_stokes} we have that
\begin{align*}
\begin{split}
\|e^{-\frac{t}{2}A }e^{-\frac{t}{2}A }{\bf u}_0\|_{L^\infty(\Om)}
 &\leq ct^{-\frac{n}{2r}}\|e^{-\frac{t}{2}}{\bf u}_0\|_{L^r(\Om)}\\
& \leq cM_0  
   t^{-\frac{ \alpha}{2}}\ln^{\delta_{\al n}} {(2+t)} ,  \quad 0<\alpha\leq n.
\end{split}
\end{align*}
Therefore we obtain  the  temporal decay rates \eqref{weighted_stokes1} for $q=\infty$ in Theorem \ref{theorem_stokes1}.


{\bf  Step 2: Pointwise estimate. 
} 

Now we will derive the pointwise  estimate for $n\geq 3$.
%
In the previous step we obtain the estimate
\[
\| {\bf u}\|_{L^\infty(\Om)}\leq cM_0(1+t)^{-\frac{\al }{2}}\ln^{\de_{\al n}}{(2+t)}.
\]
Therefore we have only to derive the estimates for $|x|\geq cR\sqrt{t+ 1}$.

According to Lemma \ref{integral_stokes0} ${\bf u}=e^{-tA}{\bf u}_0$ satisfies the integral representation
\begin{align}
\label{integral_stokes_1}
\begin{split}
{ \bf u}_i(x,t)\zeta(x)& =\int_\Omega {\bf u}_{0,i}(y)\zeta(y) \Gamma(x-y,t)dy\\
& \quad -\int_\Omega {\bf u}_0(y) \cdot [(\nabla\zeta(y))\partial_{y_i}\omega^i(x-y,t) 
 - (\nabla\zeta(y))\times \nabla\times \omega^i(x-y,t)]dy\\
&\quad +\int_\Omega {\bf u}(y,t)\cdot [(\nabla\zeta(y))\partial_{y_i}N(x-y) -(\nabla\zeta(y))\times \nabla\times (N(x-y){\mathbf e}^i)]dy\\
&\quad +\sum_{j=1}^2\int^{t}_0\int_\Omega  {\bf u}(y,\tau)\cdot R^i_j(x,y,t-\tau) dyd\tau\\
&=I_1+I_2+I_3+I_4+I_5,
\end{split}
\end{align}
where  $\zeta$ is  a smooth function with $\zeta=0$ in some neighborhood of $\partial \Om$.

{\bf $\bullet$  Estimate for $ 0 < \al \leq n-1$.}

Let us  consider the case $ 0 < \al \leq n-1$. 
 We take  $\zeta\in C^\infty_0(\Om)$ satisfying $\zeta(y)=0$ on $B_R$ and $\zeta = 1$ on $B_{2R}^c$. 
Note that $|\na^k \zeta (y) | \leq  c R^{-k}$ in $ y \in \Om_{2R} \setminus \Om_R$ and $\na^k \zeta (y) =0$ in $ (\Om_{2R} \setminus \Om_R)^c$   for $ k \in {\mathbb N}$. 



According to Lemma \ref{lemma1}
\[
|I_1 (x,t)|\leq cM_0(|x|+\sqrt{t}+1)^{-\alpha}.
\]

Throughout this section we will use the estimate of  $\nabla^k\omega, \nabla^kG$ in  Lemma \ref{lemma6_1}, and use  the temporal  estimate \eqref{weighted_stokes1} for $ q =\infty$.
Then we have 
\begin{align*}
|I_2(x,t)|& \leq c\int_{\Omega_{2R} \setminus \Om_R } ( |x-y | + \sqrt{t})^{-n +1}  |{\bf u}_0(y)|dy\\
& \leq  c(|x|+\sqrt{t})^{-n+1} \|{\bf u}_0 \|_{L^\infty (\Om)}\leq cM_0 (|x|+\sqrt{t} +1)^{-n+1}\mbox{ for } |x|  > 3R\sqrt{t+1},
\end{align*}
%
%
\begin{align*}
|I_3(x,t)|& \leq c  |x|^{-n +1}\int_{\Omega_{2R} \setminus \Om_R }  |{\bf u}(y, t )|dy  \leq c   |x|^{-n +1}    \|{\bf u }( t ) \|_{L^\infty (\Om)} \\
 & \leq c M_0    | x|^{-n +1} ( 1 + t )^{-\frac{\al}2} \leq c M_0(|x|+\sqrt{t} +1)^{-n+1}\mbox{ for } |x|  > 3R\sqrt{t+1}.
\end{align*}
Moreover, using    Lemma  \ref{1220-2} for  the estimate of $\int^t_{0}(1+\tau)^{-\frac{\al}{2}}d\tau$ we   have  
\begin{align*}
\begin{split}
| I_4(x,t)| &\leq   c\int^t_0 \int_{\Om_{2R} \setminus \Om_R} (|x-y|+ \sqrt{t-\tau})^{-n+1} | {\bf u}(y,\tau)| dyd\tau\\
& \leq c  |x|^{-n+1}\int^t_{0} \|    {\bf u} (\tau)\|_{L^\infty (\Om)}   d\tau  \leq  c  |x|^{-n+1}\int^t_{0}(1+\tau)^{-\frac{\al}{2}}d\tau\\
&\leq \left\{\begin{array}{l} \vspace{2mm}
 cM_0|x|^{-n+1}, \ \al>2\\
 \vspace{2mm}
cM_0  |x|^{-n+1}\ln{t}, \ \al=2\\
\vspace{2mm}
 cM_0    |x|^{-n+1}t^{1-\frac{\al}{2}},\  \al<2
 \end{array}\right. \\
 &\leq cM_0 ( |x| +\sqrt{t} +1)^{-\al} 
\end{split}
\end{align*}
$\mbox{ for } |x|  > 3R\sqrt{t+1}.$
 By the same reasoning as for the estimate of $I_4$  we have
\begin{align*}
|I_5(x,t)|& \leq c\int^t_0 \int_{\Om_{2R} \setminus \Om_R} (|x-y|+\sqrt{t-\tau})^{-n} | {\bf u}(y,\tau)| dyd\tau\\
& \leq   c M_0  |x|^{-n}\int^t_{0}    (1 +\tau)^{-\frac{\al}2 }   d\tau\\
& \leq c M_0(|x|+\sqrt{t} +1)^{-\frac{\al}{2}} \mbox{ for } |x|  > 3R\sqrt{t+1}.
\end{align*}

Combining all the estimates, we conclude that
\begin{equation*}
|{\bf u}(x,t)|\leq c (|x|+\sqrt{t} +1 )^{-\al} \quad \mbox{for } \quad  0 \leq \al \leq n-1.
\end{equation*}
This is  the estimate  \eqref{stokes_infty1} in Theorem \ref{theorem_stokes1}  for $ 0 < \al \leq n-1$.

{\bf  $\bullet$ Estimate for $ n-1< \al \leq n$.}

Now, we consider the case $ n-1 < \al \leq n$.
 We fix $x \in B_{3R}^c$. 
We take  $\zeta_x = \zeta \in C^\infty_0 (B(x, \frac{|x|}{2}))$ satisfying $\zeta(y)=1$ on $B(x, \frac{ |x|}{4})$. 
Note that $|\na^k \zeta (y) | \leq  c |x|^{-k}$ in $ y \in  B(x, \frac{|x|}{2}) \setminus B(x, \frac{|x|}{4}) $ and $\na^k \zeta (y) =0$ in $ ( B(x, \frac{|x|}{2}) \setminus B(x, \frac{ |x|}{4}))^c$   for $ k \in {\mathbb N}$. 
Then 
\begin{align*}
|I_1(x,t)|&\leq  c\int_{B(x,\frac{|x|}{2})}|{\bf u}_{0}(y)| \Gamma(x-y,t)dy\leq 
M_0c|x|^{-\al}\int_{{\bf R}^n} \Gamma(x-y,t)dy\\
&\leq 
cM_0(|x|+\sqrt{t}+1)^{-\alpha} \mbox{ for } |x|  > 3R\sqrt{t+1},
\end{align*}
%
\begin{align*}
|I_2(x,t)|& \leq \int_{ B(x, \frac{|x|}{2}) \setminus B(x, \frac{ |x|}{4}) }  |{\bf u}_0(y)||\nabla_x\zeta(y)||\nabla \omega(x-y,t)|dy\\
&\leq 
cM_0  |x|^{-\al-1} \int_{ B(x, \frac{|x|}{2}) \setminus B(x, \frac{ |x|}{4}) }  (|x-y|+\sqrt{t})^{-n+1}dy\\
& \leq   cM_0|x|^{-\alpha} 
\leq cM_0(|x|+\sqrt{t}+1)^{-\al } \mbox{ for } |x|  > 3R\sqrt{t+1}.
\end{align*}
 
Let $\de > 0$ be small. Take $q_\de$ with $ \frac{n}\al < q_\de = \frac{n}{\al -\de}$. 
Use the  temporal  estimate \eqref{weighted_stokes1} for $ q_\de$. 
\begin{align*}
|I_3(x,t)|&\leq c\int_{ B(x, \frac{|x|}{2}) \setminus B(x, \frac{ |x|}{4}) } |{\bf u}(y,t)||\nabla \zeta(y)||\nabla N(x-y)|dy\\
& \leq c  |x|^{-n} \int_{ B(x, \frac{|x|}{2}) \setminus B(x, \frac{ |x|}{4}) }  |{\bf u }(y,t)|dy\\
& \leq c  |x|^{-\frac{n}{q_\de} }   \|{\bf u }(t) \|_{L^{q_\de} (\Om)} \leq  c M_0  |x |^{-\frac{n}{q_\de} } (1+t)^{-\frac{\al  }{2}+\frac{n}{2q_\de}}\ln^{\delta_{\al n}}{(2+t)}\\
&\leq c_\delta M_0(|x|+\sqrt{t}+1)^{-\al+\de}\mbox{ for } |x|  > 3R\sqrt{t+1}.
\end{align*}
Moreover, use  Lemma  \ref{1220-2} for  the estimate of  $\int^t_{0} (1 +\tau)^{-\frac{\al }2 +\frac{n}{2q_\de}} \ln^{\delta_{\al n}}{(2+\tau)}d\tau$, then we have  
\begin{align*}
\begin{split}
| I_4(x,t)|&\leq c \int^t_{0} \int_{ B(x, \frac{|x|}{2}) \setminus B(x, \frac{ |x|}{4}) }  |{\bf u}(y,\tau)|\Big(  |\nabla^2\zeta(y)||\nabla^2\omega(x-y,t-\tau)|\\
&\qquad\qquad+|\nabla^3\zeta(y)||\nabla \omega(x-y,t-\tau)|  \Big)dy   d\tau \\
 &  \leq c  |x|^{-n-2}\int^t_{0} \int_{ B(x, \frac{|x|}{2}) \setminus B(x, \frac{ |x|}{4}) }  |{\bf u}(y,\tau)|dy   d\tau \\
&  \leq c  |x|^{-\frac{n}{q_\de}-2}\int^t_{0}\|{\bf u}(\tau)\|_{L^{q_\de} (\Om)}    d\tau \\
&  \leq  cM_0 |x|^{-\frac{n }{q_\de}-2}   \int^t_0 (1 +\tau)^{-\frac{\al }2 +\frac{n}{2q_\de}} \ln^{\delta_{\al n}}{(2+t)}d\tau\\
& \leq c_\delta M_0   |x|^{-\frac{n }{q_\de}-2} (t+1)^{-\frac{\al}2 +\frac{n}{2q_\de} +1+\frac{\delta}{2}} \\
& \leq c_\delta M_0  (  |x| +\sqrt{t} +1)^{-\al +\delta } \mbox{ for } |x|  > 3R\sqrt{t+1}.
\end{split}
\end{align*}


By the same argument as for  $I_4$, we have 
\begin{align*}
|I_5(x,t)| 
& \leq c_\delta  M_0  (  |x| +\sqrt{t} +1)^{-\al + \delta } \mbox{ for } |x|  > 3R\sqrt{t+1}.
\end{align*}
 
Combining all the estimates, we obtain
\begin{equation*}
|{\bf u}(x,t)|\leq c_\delta (  |x| +\sqrt{t} +1)^{-\al +\de } \quad \mbox{for } \quad  n-1 <  \al \leq n.
\end{equation*}
Summing all the estimates, we obtain the estimate \eqref{stokes_infty1} of Theorem \ref{theorem_stokes1}.

\section{Proof of Theorem \ref{theorem1}}
\setcounter{equation}{0}
\label{external-force}

The following lemma will be used in proving  the solvability of the nonlinear problem.

\begin{lemm}
\label{bilinear1}
Let  $1< r \leq q<\infty$(or $1=r<q\leq \infty)$.

If $\frac{n}{n-1}\leq  r $, then
\[
\|
e^{-tA}{\mathbb P}\mbox{\rm div}{\mathcal F}\|_{L^q} \leq ct^{-\frac{1}{2}-\frac{n}{2}(\frac{1}{r}-\frac{1}{q})}\|{\mathcal F}\|_{L^r},  \quad t>0
\]
and if $1<r\leq\frac{n}{n-1},$ then
 \[
\|
e^{-tA}{\mathbb P}\mbox{\rm div}{\mathcal F}\|_{L^q} \leq ct^{-\frac{n}{2}(1-\frac{1}{q})}\|{\mathcal F}\|_{L^r},  \quad   t > 1.
\]

\end{lemm}

\begin{proof}

For  $\varphi\in C^\infty_{0,\sigma}(\Om)$ the following identity holds 
\begin{align*}
<e^{-tA}{\mathbb P}\mbox{\rm div}{\mathcal F},\varphi> & =-<{\mathcal F}, \nabla e^{-tA}\varphi> \leq \| {\mathcal F} \|_{L^r (\Om)}  \|\nabla e^{-tA}\varphi\|_{L^{r'}(\Om)}.
\end{align*}
According to Proposition \ref{LqLr_stokes}, 
 for $1< q'\leq r'\leq n$   or $1\leq q'<r'\leq n$ it holds  
\[\|\nabla e^{-tA}\varphi\|_{L^{r'}(\Om)}\leq ct^{-\frac{1}{2}-\frac{n}{2}(\frac{1}{q'}-\frac{1}{r'})}\|\varphi\|_{L^{q'}(\Om)}, \quad t>0,
\]
and  for $1<q'\leq r'$ and  $ \ n\leq r'<\infty$,  it holds  
\[\|\nabla e^{-tA}\varphi\|_{L^{r'}(\Om)}\leq ct^{-\frac{n}{2q'}}\|\varphi\|_{L^{q'}(\Om)},\quad  t\geq 1.
\]
By the duality argument this leads to the completion of  the proof of Lemma \ref{bilinear1}.

\end{proof}

Now, we construct approximate solutions by the successive argument:
Let ${\bf u}^{(0)} = e^{-tA}{\bf u}_0$. After obtaining ${\bf u}^{(1)},\cdots, {\bf u}^{(m)}$  
construct 
${\bf u}^{(m+1)}$ defined by
\begin{equation}
\label{approx_n}
{\bf u}^{(m+1)}(t):=e^{-tA}{\bf u}_0-\int^t_0e^{-(t-\tau)A}{\mathbb P}\mbox{\rm div}({\bf u}^{(m)}\otimes {\bf u}^{(m)})(\tau) d\tau,
\end{equation}
Then there is $p^{(m+1)}$ so that  $({\bf u}^{(m+1)}, p^{(m+1)})$ satisfies the equation
\begin{align}
\label{e1external-2}
\left\{\begin{array}{l}\vspace{2mm}
\partial_t{\bf u}^{(m+1)}-\Delta {\bf u}^{(m+1)} +\nabla p^{(m+1)}=-{\rm div} \,  ( {\bf u}^{(m )}\otimes {\bf u}^{(m)} ) \,\, \mbox{ in } \,\, \Omega\times (0,\infty),\\
\vspace{2mm}
\mbox{\rm div}\, {\bf u}^{(m+1)}=0 \,\, \mbox{ in }\,\, \Omega\times (0,\infty),\\
\vspace{2mm}
{\bf u}^{(m+1)}|_{\pa \Om} =0,\\
\vspace{2mm}
 \lim_{|x|\rightarrow \infty}{\bf u}^{(m+1)}(x,t)=  \,\, 0\mbox{ for }t>0,\\
{\bf u}^{(m+1)} (x,0)={\bf u}_0\,\, \mbox{ in } \,\, \Omega.
\end{array}\right.
\end{align}

We introduce function space ${\mathcal X}(\al, q)$  
by \begin{align}\label{0310-1}
\begin{split}
{\mathcal X}(\al, q) :  =\Big\{ {\bf u } ;\  & 
(1+t)^{\frac{\al}{2}-\frac{n}{2q}} \ln^{\delta_{\al n}} {(2+t)}  {\bf u}(t)\in BC(0,\infty;L^q(\Om)),\\
&(1+t)^{\frac{\al}{2}}  \ln^{\delta_{\al n}} {(2+t)}  {\bf u}(t)\in BC(0,\infty;L^\infty(\Om)),\\
&t^{\frac{1}{2}}(1+t)^{\frac{\al}{2}-\frac{1}{2}}  \ln^{\delta_{\al n}} {(2+t)}   \nabla {\bf u}(t)\in BC(0,\infty;L^n(\Om))
\Big\}
\end{split}
\end{align}
endowed with the norm
\begin{align*}
\|{\bf u}\|_{{\mathcal X}(\al, q)} : =&\sup_{0<t<\infty}(1+t)^{\frac{\al}{2}}  \ln^{\delta_{\al n}} {(2+t)}   \|{\bf u}(t)\|_{ L^\infty(\Om)}
\\
&+ \sup_{0<t<\infty}(1+t)^{\frac{\al}{2}-\frac{n}{2q}}  \ln^{\delta_{\al n}} {(2+t)}   \|{\bf u}(t)\|_{  L^q(\Om)}\\
&+ \sup_{0<t<\infty}t^{\frac{1}{2}}(1+t)^{\frac{\al}{2}-\frac{1}{2}}  \ln^{\delta_{\al n}} {(2+t)}   \| \nabla {\bf u}(t)\|_{L^{n}(\Om)}.
\end{align*}

\begin{rem}
We note that 
\begin{align}\label{0211-1}
 \|u(t)\|_{L^r(\Om)}\leq (1+t)^{-\frac{\al}{2}+\frac{n}{2r}} \ln^{\delta_{\al n}} {(2+t)}\|u\|_{{\mathcal X}(\al,q)}   \quad q\leq r\leq \infty.
\end{align}
This can be done by interpolating the two estimate
\begin{align*}
\|u(t)\|_{L^q(\Om)} & \leq (1+t)^{-\frac{\al}{2}+\frac{n}{2q}} \ln^{\delta_{\al n}} {(2+t)} \|u\|_{{\mathcal X}(\al,q)} ,
\\
\|u(t)\|_{L^\infty(\Om)} & \leq (1+t)^{-\frac{\al}{2}}  \ln^{\delta_{\al n}} {(2+t)}\|u\|_{{\mathcal X}(\al,q)}. 
\end{align*}
\end{rem}

Fix   $q$ with $\frac{n}{\al}<q<\frac{n}{\al-1}$.  Below, we will show the uniform boundedness of $\{ {\bf u}^{(m)}\}$ in ${\mathcal X}(\al,q)$.

 \subsection{Uniform boundedness. 
 }
 
According to the estimate \eqref{weighted_stokes1} and \eqref{weighted_stokes2} of Theorem \ref{theorem_stokes1} and  Corollary \ref{coro_stokes1}, there is $c_{1, \al, q}, \,\,\, c_{2,\al, q} >0$ such that
\begin{align*}
\|e^{-tA}{\bf u}_0\|_{L^q(\Om)} & \leq  c_{1,\al,q}M_0
(1+t)^{-\frac{\al}{2}+\frac{n}{2q}} \ln^{\delta_{\al n}} {(2+t)} , \quad \frac{n}{\al}<q\leq \infty,\   0 < \al \leq  n,\\
\|\nabla e^{-tA} {\bf u}_0\|_{L^n(\Om)} & \leq c_{2,\al,q}M_0   
t^{-\frac{1}{2}}(1+t)^{-\frac{\alpha }{2}+\frac{1}{2}} \ln^{\delta_{\al n}} {(2+t)} ,\quad      0 < \al \leq  n.
\end{align*}
Therefore 
\begin{equation}
\label{4.4}
\|{\bf u}^{(0)}\|_{{\mathcal X} (\al,  q)}\leq c_{0,\al,q}M_0\mbox{ for any }q\in (\frac{n}{\al},\infty).
\end{equation}

We assume that 
\[
\|{\bf u}^{(k)}\|_{{\mathcal X}(\al,q)}\leq M,\ k=1,\cdots, m.\]
 Below we will show that there is $M:=c_{*}M_0$ for some positive constant $c_{*}$  so that 
 \[
\|{\bf u}^{(m+1)}\|_{{\mathcal X}(\al,q)}\leq M\]
as far as $M_0$ is small enough.

 

{\bf Step 1.  
}
In this step, we  will derive the estimate of   $\|{\bf u}^{(m+1)}\|_{L^q(\Om)}$,  $\frac{n}{\al}<q<\frac{n}{\al-1}$.
 
 Observe that from \eqref{4.4}
\begin{align}
\label{eq*}
\|{\bf u}^{(m+1)}\|_{L^q(\Om)}& \leq 
c_{0,\al,q}M_0 + 
 \|\int^t_0   e^{-(t -\tau)A} \Big( {\mathbb P} \big( {\rm div} \, {\bf u}^{(m)}(\tau)\otimes  {\bf u}^{(m)}(\tau)  \big) \Big) d\tau\|_{L^q(\Om)}.
 \end{align}

(i) Let $1\leq \al \leq n-1$ and $\frac{n}{\al}<q<\frac{n}{\al-1}$. Then $\frac{1}{n}+\frac{1}{q}<1$ and $-\frac{1}{2}<-\frac{\al}{2}+\frac{n}{2q}$. 
Take $ 1 < q_1 < \min(q, n)$ such that $\frac1{q_1} = \frac1n + \frac1q$.  According to Proposition \ref{LqLr_stokes}, we have 
\begin{align}\label{kind2}
\begin{split}
 &\int^{t}_0\|  e^{-(t -\tau)A} \Big( {\mathbb P}  {\rm div} \,  \big( {\bf u}^{(m)}(\tau)\otimes  {\bf u}^{(m)}(\tau)  \big) \Big) \|_{L^{q}(\Om)} d\tau \\
  &\int^{t}_0 (t -\tau)^{   -\frac{n}{2q_1} +\frac{n}{2q}} \|   {\rm div} \,  \big(   {\bf u}^{(m)}(\tau)\otimes  {\bf u}^{(m)}(\tau)  \big)  \|_{L^{q_1}(\Om)} d\tau \\
 &\leq c\int^{t}_0(t-\tau)^{-\frac{1}{2}}\|\nabla {\bf u}^{(m)}\|_{L^n(\Om)}\|{\bf u}^{(m)}\|_{L^q(\Om)}d\tau\\
&\leq cM^2 \int^{t}_0(t-\tau)^{-\frac{1}{2}}\tau^{-\frac{1}{2}}(1+\tau)^{-\al+\frac{1}{2}+\frac{n}{2q}} d\tau\\
& \leq c_{\al,q}M^2(1+t)^{-\frac{\al}{2}+\frac{n}{2q}}. 
\end{split}
\end{align}
Here we used the estimate
\begin{align*}
\int^{t}_0(t-\tau)^{-\frac{1}{2}}\tau^{-\frac{1}{2}}(1+\tau)^{-\al+\frac{1}{2}+\frac{n}{2q}} d\tau
& \leq c\left\{\begin{array}{l} \vspace{2mm}
1, \quad t\leq 1\\
t^{-\frac{1}{2}}\ln (1+t)+ct^{-\al+\frac{n}{2q}+\frac{1}{2}}, \quad  t\geq 1,
\end{array}\right.\\
& \leq c(1+t)^{-\frac{\al}{2}+\frac{n}{2q}},  \quad  \al\geq 1,  \quad \frac{n}{\al}<q<\frac{n}{\al-1}.
\end{align*}

(ii) Let $n-1<\al\leq n$ and $\frac{n}{\al}<q<\frac{n}{\al-1}$.   Note that $-\al+\frac{n}{2q}   < -1$  because of  $ \frac{n}{2(\al -1)} < \frac{n}{\al}$ for  $2 < \al$.  According to Lemma \ref{bilinear1}  and \eqref{0211-1}, we have 
 \begin{align}  \label{kind22} 
 \begin{split}
  &\int^t_{0} \|  e^{-(t -\tau)A} \Big( {\mathbb P} \big( {\rm div} \, {\bf u}^{(m)}(\tau)\otimes  {\bf u}^{(m)}(\tau)  \big) \Big) \|_{L^{q}(\Om)} d\tau\\
 & \leq c   \int^t_{0}(t-\tau)^{-\frac12 }\|  {\bf u}^{(m)}(\tau) )\|_{L^{2q} (\Om)}^2 d\tau\\
 & \leq  c  \|{\bf u}^{(m)}\|_{{\mathcal X}(\al,q)}^2   \int^t_{0}(t-\tau)^{-\frac12 } (1+\tau)^{-\al+\frac{n}{2q}  }  \ln^{2\de_{\al n}} {(2+\tau)} d\tau\\
&\leq 
 c_{\al,q}M^2(1+t)^{-\frac{\al}{2}+\frac{n}{2q}}  \ln^{\de_{\al n}} {(2+t)} .
 \end{split}
\end{align}
Here we used the estiamte
\begin{align*}
\int^t_{0} (t-\tau)^{-\frac12} (1+\tau)^{-\al+\frac{n}{2q} }  \ln^{2\de_{\al n}} {  (2+\tau) }d\tau\leq 
c\left\{\begin{array}{l} \vspace{2mm}
1 \quad  t\leq 1,\\
t^{-\frac12  } + t^{-\al+\frac{1}{2}+\frac{n}{2q}}  \ln^{2 \de_{\al n}}  {(2+t)}  \quad  t\geq 1.
\end{array}\right.
\end{align*}
Let $ 1\leq \al\leq n$. For any $q$ with  $\frac{n}{\al}<q< \frac{n}{\al-1}$,   \eqref{kind2}-\eqref{kind22} lead to the 
conclusion that
\begin{equation}
\label{kind66}
\|{\bf u}^{(m+1)}(t)\|_{L^q(\Om)}\leq c_{1,\al,q}(M_0+M^2)(1+t)^{-\frac{\al}{2}+\frac{n}{2q}}   \ln^{\de_{\al n}}  {(2+t)} .
 \end{equation}

{\bf Step 2.} Second, we derive the estimate of $\|{\bf u}^{(m+1)}\|_{L^\infty(\Om)}$. Let $\phi\in C^\infty_{0,\sigma}(\Om)$.
Then the following identity holds:
\begin{align*}
<{\bf u}^{(m+1)}(t),\phi_0>&=<{\bf u}^{(m+1)}(\frac{t}{2}),e^{-\frac{t}{2}A}\phi_0>\\
&
-\int^t_{\frac{t}{2}}<{\mathbb P}\mbox{\rm div}({\bf u}^{(m)}(\tau)\otimes {\bf u}^{(m)}(\tau)),e^{-(t-\tau)A}\phi_0>d\tau=I+II.
\end{align*}
Acoodrding to \eqref{kind66}
\begin{align}
\label{kind7}
\begin{split}
 I&\leq c\|{\bf u}^{(m+1)}(\frac{t}{2})\|_{L^q(\Om)}\|e^{-\frac{t}{2}A}\phi_0\|_{L^{q'}(\Om)}\\
&
\leq c(M_0+M^2)(1+t)^{-\frac{\al}{2}+\frac{n}{2q}}   \ln^{\de_{\al n}} {(2+t)}  (1+t)^{-\frac{n}{2}(1-\frac{1}{q'})} \|\phi_0\|_{L^1(\Om)}
\\
&=c(M_0+M^2)(1+t)^{-\frac{\al}{2}}   \ln^{\de_{\al n}} {(2+t)}  \|\phi_0\|_{L^1(\Om)}.
\end{split}
\end{align}
According to Proposition \ref{LqLr_stokes}
\begin{align}
\label{kind8}
\begin{split}
  II&\leq c\int^t_{\frac{t}{2}}\|\mbox{\rm div}({\bf u}^{(m)}(\tau)\otimes {\bf u}^{(m)}(\tau))\|_{L^n(\Om)}\|e^{-A(t-\tau)A}\phi_0\|_{L^{n'}(\Om)}d\tau\\
 &\leq c\int^t_0(t-\tau)^{-\frac12   }\| \na {\bf u}^{(m)}(\tau) )\|_{L^n (\Om)} \|{\bf u}^{(m)}(\tau) )\|_{L^{\infty} (\Om)}\|\phi_0\|_{L^1(\Om)} d\tau\\
 &\leq c\|{\bf u}^{(m)}\|_{{\mathcal X}(\al,q)}^2\|\phi_0\|_{L^1(\Om)}\int^t_{\frac{t}{2}}(1+\tau)^{-\al+\frac{1}{2}}\tau^{-\frac{1}{2}}(t-\tau)^{-\frac{1}{2}}   \ln^{2 \de_{\al n}}  {(2+t)} d\tau\\
&\leq cM^2(1+t)^{-\al+\frac{1}{2}}   \ln^{2 \de_{\al n}}  {(2+t)} \\
& \leq cM^2(1+t)^{-\frac{\al}{2}}  \ln^{\de_{\al n}} {(2+t)} , \quad  \al\geq 1.
\end{split}
\end{align}

For $ 1\leq \al\leq  n$ and $ \frac{n}{\al}<q<\frac{n}{\al-1}$,     \eqref{4.4}, \eqref{eq*},  \eqref{kind7} and \eqref{kind8} lead to the estimate
\begin{equation}
\label{kind9}
\|{\bf u}^{(m+1)}(t)\|_{L^\infty(\Om)}\leq c_{\al,q}(M_0+M^2)(1+t)^{-\frac{\al}{2}}  (  \ln  {(2+t )} )^{\de_{\al n}}.
 \end{equation}
 
{\bf Step 3.}
Now,  derive the estimates of   
 $\|\nabla {\bf u}^{(m+1)}(t)\|_{L^n(\Om)}$ . 

Observe that from \eqref{4.4}
\begin{align}
\label{eq*}
\|\nabla {\bf u}^{(m+1)}\|_{L^n(\Om)}& \leq 
c_{0,\al,q}M_0 + 
 \|\int^t_0   \nabla e^{-(t -\tau)A} \Big( {\mathbb P} \big( {\rm div} \, {\bf u}^{(m)}(\tau)\otimes  {\bf u}^{(m)}(\tau)  \big) \Big) d\tau\|_{L^n(\Om)}.
 \end{align}

(i) Let $t\leq 1$. Then, from  Proposition \ref{LqLr_stokes}, we have 
\begin{align}\label{kind55}
\begin{split}
 &\int^t_{0} \| \na e^{-(t -\tau)A} \Big( {\mathbb P} \big( {\rm div} \, {\bf u}^{(m)}(\tau)\otimes  {\bf u}^{(m)}(\tau)  \big) \Big) \|_{L^{n}(\Om)} d\tau\\
 & \leq  c \int^t_{0} ( t -\tau)^{-\frac12}  \|  \na {\bf u}^{(m)}(\tau)  \|_{L^{n}(\Om)}  \|   {\bf u}^{(m)}(\tau) \|_{L^\infty(\Om)} d\tau\\
&  \leq  c \|{\bf u}^{(m)}\|_{{\mathcal X}(\al,q)}^2  \int^t_{0} ( t -\tau)^{-\frac12 } \tau^{-\frac12} (1 +\tau)^{-\al +\frac12}    ( \ln {(2+\tau)} )^{2\de_{\al n}} d\tau\\
& \leq cM^2.
\end{split}
\end{align}

(ii) Let $t> 1$.
According to Proposition \ref{LqLr_stokes},
\begin{align}\label{kind5} 
\begin{split}
 &\int^t_{\frac{t}2} \| \na e^{-(t -\tau)A} \Big( {\mathbb P} \big( {\rm div} \, {\bf u}^{(m)}(\tau)\otimes  {\bf u}^{(m)}(\tau)  \big) \Big) \|_{L^{n}(\Om)} d\tau\\
 & \leq  c \int^t_{\frac{t}2} ( t -\tau)^{-\frac12}  \|  \na {\bf u}^{(m)}(\tau)  \|_{L^{n}(\Om)}  \|   {\bf u}^{(m)}(\tau) \|_{L^\infty(\Om)} d\tau\\
 & \leq  c \|{\bf u}^{(m)}\|_{{\mathcal X}(\al,q)}^2  \int^t_{\frac{t}2} ( t -\tau)^{-\frac12 } \tau^{-\frac12} (1 +\tau)^{-\al +\frac12}    ( \ln {(2+\tau)} )^{2\de_{\al n}} d\tau\\
 & \leq  cM^2   t^{-\frac12} (1+ t)^{-\al +\frac12}    ( \ln {(2+\tau)} )^{2\de_{\al n}}\int^t_{\frac{t}2} ( t -\tau)^{-\frac12 }  d\tau\\
& \leq  c M^2 (1+ t)^{-\al +\frac12}    ( \ln {(2+\tau)} )^{2\de_{\al n}}\\
&\leq  cM^2 (1+ t)^{-\frac{\al}2 }    ( \ln {(2+\tau)} )^{2\de_{\al n}}, \quad \al\geq 1 . 
\end{split}
\end{align}
Likewise,
\begin{align}
\label{kind11}
\begin{split}
 &\int^{\frac{t}{2}}_{0} \| \na e^{-(t -\tau)A} \Big( {\mathbb P} \big( {\rm div} \, {\bf u}^{(m)}(\tau)\otimes  {\bf u}^{(m)}(\tau)  \big) \Big) \|_{L^{n}(\Om)} d\tau\\
 & \leq  c \int^{\frac{t}{2}}_{0} ( t -\tau)^{-\frac{1}{2}-\frac{n}{2q}}  \|  \na {\bf u}^{(m)}(\tau)  \|_{L^{n}(\Om)}  \|   {\bf u}^{(m)}(\tau) \|_{L^q(\Om)} d\tau\\
 & \leq  c \|{\bf u}^{(m)}\|_{{\mathcal X}(\al,q)}^2  \int^{\frac{t}{2}}_{0} ( t -\tau)^{-\frac12-\frac{n}{2q} } \tau^{-\frac12} (1 +\tau)^{-\al +\frac12+\frac{n}{2q}}    ( \ln {(2+\tau)} )^{2\de_{\al n}} d\tau\\
& \leq  cM^2 t^{-\frac{1}{2}-\frac{n}{2q}}\int^{\frac{t}{2}}_{0}\tau^{-\frac12} (1 +\tau)^{-\al +\frac12+\frac{n}{2q}}    ( \ln {(2+\tau)} )^{2\de_{\al n}} d\tau\\
& \leq cM^2t^{-\frac{\al}{2}}   ( \ln {(2+\tau)} )^{\de_{\al n}},\quad  \al\geq 1 . 
\end{split}
\end{align}
 
For $ 1\leq \al\leq n$ and $ \frac{n}{\al}<q<\frac{n}{\al-1}$, 
\eqref{4.4}, \eqref{eq*}, \eqref{kind55}-\eqref{kind11} lead to the estimate
\begin{align}
\label{kind10}
\|\nabla {\bf u}^{(m+1)}(t)\|_{L^n(\Om)}\leq c(M_0+M^2)t^{-\frac{1}{2}}(1+t)^{-\frac{\al}{2}+\frac{1}{2}}    ( \ln {(2+\tau)} )^{\de_{\al n}}.
\end{align}

{\bf Step 4.}  Combining \eqref{kind66}, \eqref{kind9} and \eqref{kind10}, we conclude that
\begin{equation}
\label{1224-1}
\|{\bf u}^{(m+1)}\|_{{\mathcal X}(\al, q)}\leq C_{0} M_0+C_1M^2.
 \end{equation}

If we choose  
\begin{equation}
	M=2C_0M_0  \quad \mbox{and } \quad   2C_0C_1M_0\leq \frac{1}{2},
\end{equation}
 then
 \[
\|{\bf u}^{(m+1)}\|_{{\mathcal X}(\al, q)}\leq   C_0 M_0 + C_1M^2\leq  \frac12 M + \frac12 M = M.
\]
Hence, we have  
\begin{equation}
\label{4.17}
\|{\bf u}^{(m)}\|_{{\mathcal X} (\al, q)}\leq M (:=2C_0M_0)\quad \mbox{ for all } m.
\end{equation}

\subsection{Convergence  in ${\mathcal X}(\al,q)$}

Let ${\bf V}^{(m)}:={\bf u}^{(m+1)}-{\bf u}^{(m)}$.  Then from  \eqref{approx_n} we have 
\begin{equation}
\label{approx_n2}
{\bf V}^{(m)}(t):=\int^t_0e^{-(t-\tau)A}{\mathbb P}\mbox{\rm div}({\bf V}^{(m-1)}\otimes {\bf u}^{(m)}+{\bf u}^{(m-1)}\otimes  {\bf V}^{(m-1)})(\tau) d\tau.
\end{equation}
By the same argument appeared in the process of the proof of \eqref{1224-1}, we can obtain the following estimate
\[
\|{\bf V}^{(m)}\|_{{\mathcal X}(\al, q)}\leq C_2 \big( \|{\bf u}^{(m)}\|_{{\mathcal X}(\al, q)}  + \|{\bf u}^{(m-1)}\|_{{\mathcal X}(\al, q)} \big)\|{\bf V}^{(m-1)}\|_{{\mathcal X}(\al, q)} \leq 2 C_2 M\|{\bf V}^{(m-1)}\|_{{\mathcal X}(\al, q)}
\]
for $\frac{n}{\al}<q<\frac{n}{\al-1}, 1\leq \al\leq n$.

If we choose $M$ so small that $2C_2M < \frac12$, then

\[
\|{\bf V}^{(m)}\|_{{\mathcal X}(\al, q)}\leq  \frac{1}{2}
 \|{\bf V}^{(m-1)}\|_{{\mathcal X}(\al, q)}.
\]
This leads to the convergence of the series $\sum_{m=0}^\infty{\bf V}^{(m)}$.
This implies the convergence of the sequence $\{{\bf u}^{(m)}\}$ since 
\[
\sum_{m=n}^\infty{\bf V}^{(m)}={\bf u}^{(m)}-{\bf u}^{(0)}.\]
Let ${\bf u}$ be the limit of ${\bf u}^{(m)}$.
Then
${\bf u}$ satisfies the estimate \eqref{navier_weight_time}:

It is easy to check that
\begin{align}\label{240118-1}
{\bf u}=e^{-tA}{\bf u}_0-\int^t_0e^{-(t-\tau)A}{\mathbb P}\mbox{\rm div}({\bf u}\otimes {\bf u})(\tau) d\tau
\end{align}
 in the sense of distributions.

\subsection{Uniqueness in ${\mathcal X}(\al,q)$}

In the previous section we obtain the  solution in ${\mathcal X}(\al, q)$ for  $\al\geq 1$ and    $\frac{n}{\al}<q<\frac{n}{\al-1}$. 
In this section we show the uniqueness in the solution class ${\mathcal X}(\al,q)$.

If $q>n$, then take $q_0=q$. Otherwise,  take  $q_0\in (n,\infty]$.  According to \eqref{0211-1}, $ {\mathcal X}(\al,q)\subset  L^\infty(0,T;L^{q_0}(\Om))$ for $q_0>q$.  The uniqueness of the Navier-Stokes flow in $L^{s_0}(0,T;L^{q_0}(\Om))$,  $\frac{2}{s}+\frac{n}{q_0}\leq 1$  can be found in many literature. 
Nonetheless, for the sake of the completeness 
we  will give the proof  that the solution is unique in the class  $L^\infty(0,T; L^{q_0}(\Om))$ for any $T\leq \infty$. 
%

Let $\tilde{\bf u}\in  L^\infty(0,T; L^{q_0}(\Om))$ be another  mild solution of \eqref{e1}  in the form \eqref{240118-1}. 
Let  ${\bf V}={\bf u}-\tilde{\bf u}$. Then
\[
{\bf V}(t)=-\int^t_0e^{-(t-\tau)A}{\mathbb P}\mbox{\rm div}( {\bf V}\otimes {\bf u}+\tilde{\bf u}\otimes {\bf V})(\tau) d\tau, \quad \ t>0.
\]

Let $\sup_{0<t\leq T}
\|{\bf u}(t)\|_{L^{q_0}(\Om)}=M, 
$ and $\sup_{0<t\leq T}
\|\tilde{\bf u}(t)\|_{L^{q_0}(\Om)}
=M_1.$
Applying Lemma \ref{bilinear1} we have
\begin{align*}
\|{\bf V}(t)\|_{L^{q_0}(\Om)}\leq c\int^t_0(t-\tau)^{-\frac{1}{2}-\frac{n}{2q_0}}(\|{\bf u}(\tau)\|_{L^{q_0}(\Om)}+\|\tilde{\bf u}(\tau)\|_{L^{q_0}(\Om)})\|{\bf V}(\tau)\|_{L^{q_0}(\Om)}d\tau
\\
\leq c(M+M_1)\int^t_0(t-\tau)^{-\frac{1}{2}-\frac{n}{2q_0}}
\|{\bf V}(\tau)\|_{L^{q_0}(\Om)}d\tau.
\end{align*}
This leads to the estimate
\begin{align*}
\sup_{0<\tau<t_0}\|{\bf V}(\tau)\|_{L^{q_0}(\Om)}
\leq c_*(M+M_1)   t_0^{\frac{1}{2}-\frac{n}{2q_0}}\sup_{0<\tau<t_0}
\|{\bf V}(\tau)\|_{L^{q_0}(\Om)}. 
\end{align*}
If $t_0$ is small enough with $c_*(M+M_1)t_0^{\frac{1}{2}-\frac{n}{2q_0}}<1$, then 
\begin{align*}
\sup_{0<\tau<t_0}\|{\bf V}(\tau)\|_{L^{q_0}(\Om)}=0,
\mbox{ that is, }{\bf u}(\tau)=\tilde{\bf u}(\tau),\ 0\leq \tau\leq t_0.\end{align*}

This implies
\[
{\bf V}(t)=-\int^t_{t_0}e^{-(t-\tau)A}{\mathbb P}\mbox{\rm div}({\bf V}\otimes {\bf u}+\tilde{\bf u}\otimes {\bf V})(\tau) d\tau,\ t>t_0.
\]
By the same procedure  we have 
\begin{align*}
\|{\bf V}(t)\|_{L^{q_0}(\Om)}\leq c\int^t_{t_0}(t-\tau)^{-\frac{1}{2}-\frac{n}{2q_0}}(\|{\bf u}(\tau)\|_{L^{q_0}(\Om)}+\|\tilde{\bf u}(\tau)\|_{L^{q_0}(\Om)})\|{\bf V}(\tau)\|_{L^{q_0}(\Om)}d\tau
\\
\leq c(M+M_1)\int^t_{t_0}(t-\tau)^{-\frac{1}{2}-\frac{n}{2q_0}}
\|{\bf V}(\tau)\|_{L^{q_0}(\Om)}d\tau.
\end{align*}
This leads to the estimate
\begin{align*}
\sup_{t_0<\tau<t_1}\|{\bf V}(\tau)\|_{L^{q_0}(\Om)}
\leq c_*(M+M_1)   (t_1-t_0)^{\frac{1}{2}-\frac{n}{2q_0}}\sup_{t_0<\tau<t_1}
\|{\bf V}(\tau)\|_{L^q(\Om)}. 
\end{align*}
If $t_0$ is small enough with $c_*(M+M_1)(t_1-t_0)^{\frac{1}{2}-\frac{n}{2q_0}}<1$, then 
\begin{align*}
\sup_{t_0<\tau<t_1}\|{\bf V}(\tau)\|_{L^q(\Om)}=0,
\mbox{ that is, }{\bf u}(\tau)=\tilde{\bf u}(\tau),\ t_0\leq \tau\leq t_1.
\end{align*}

We iterate this procedure infinitely, then we can conclude that ${\bf u}(t)=\tilde{\bf u}(t)$ for all $t>0$.
This implies ${\bf u} \equiv \tilde {\bf u}$. Therefore, we complete the proof of uniqueness.

\begin{rem}

From the representation \eqref{240118-1} of ${\bf u}$, it holds that
\[
\partial_t {\bf u}+A{\bf u}=-{\mathbb P}\mbox{\rm div}({\bf u}\otimes {\bf u}) \mbox{ in the sense of distributions}
 .\]

Let $u^{[1]}=e^{-tA}{\bf u}_0.$ Observe that
$D_x^2{\bf u}^{[1]}, \partial_t{\bf u}^{[1]} \in C(0,\infty;L^q(\Om))$  and there is $ p^{[1]}$ with $\nabla p^{[1]}\in C(0,\infty;G_q(\Om))$ satisfying that
\[
\partial_tu^{(1)}-\De u^{(1)}+\nabla p^{(1)}=0.
\]

Set $f=-\mbox{\rm div}({\bf u}\otimes {\bf u})$. 
 Set $
{u}^{[2]}=\int^t_0e^{-(t-\tau)A}{\mathbb P}f(\tau) d\tau.$ Then $\partial_tu^{[2]}+Au^{[2]}={\mathbb P}f$.
Observe that  ${f} \in C^\infty(0,\infty;L^n(\Om)) $ with 
$\|f\|_{L^n(\Om)}\leq c(1+t)^{-\al+\frac{1}{2}}t^{-\frac{1}{2}}M_0.$ This implies that ${\mathbb P}f\in L^s(0,T;J_n(\Om))$ for any $s<2$ and $T<\infty$.
According to $L^sL^r$ maximal regularity theorem in \cite{GS}   $D_x^2u^{[2]}, \partial_tu^{[2]}\in L^s(0,T;L^n(\Om)$ and there is $p^{[2]}$ with $\nabla p^{[2]}\in L^s(0,T;G_n(\Om))$ with 
\[
\partial_tu^{[2]}-\Delta u^{[2]}+\nabla p^{[2]}={\mathbb P}f.\]

Therefore ${\bf u}=u^{[1]}+u^{[2]}$ is a strong solution of \eqref{e1} with associated pressure  $p=p^{[1]}+p^{[2]}$.

\end{rem}

\section{Proof of Corollary \ref{navier_wieght_space}}

In this section we will show that ${\bf u}$ satisfies  the pointwise estimate \eqref{navier_wieght_space}. 
This can be done by deriving uniform pointwise estiamte of each sequence ${\bf u}^{(m)}$ constructed  in \eqref{approx_n}  and by lower-semi continuity.
%

According to \eqref{stokes_infty1} of Theorem \ref{theorem_stokes1} \begin{equation}
|{\bf u}^{(0)}(x,t)|\leq  \left\{\begin{array}{l} \vspace{2mm}
c_\al M_0 (1+|x|+\sqrt{t})^{-\al},\, \,  0<\al\leq n-1,\\
c_{\al,\delta} M_0 (1+|x|+\sqrt{t})^{-\al+\delta},\,\, n-1<\al\leq n.
\end{array}\right. 
\end{equation}

Assume that ${\bf u}^{(k)},k=1,\cdots, m$ satisfy the  estimate
\begin{equation}
\label{4.20}
|{\bf u}^{(k)}(x,t)|\leq \left\{\begin{array}{l}
\vspace{2mm}
 N (1+|x|+\sqrt{t})^{-\al},1\leq \al\leq n-1\\
N_\delta (1+|x|+\sqrt{t})^{-\al+\delta},n-1<\al\leq n\mbox{ for any small  }\de>0.
\end{array}\right. 
\end{equation}

First, we  will   show that ${\bf u}^{(m+1)}$ satisfies the inequality  $\eqref{4.20}_1$ for some  small $N:=C_{*}M_0$. 
Second, we will consider the case $n-1<\al\le n$. Using the fact that ${\bf u}^{(m)}$ satisfies the estimate \eqref{4.17} and $\eqref{4.20}_1$ we will show that  
${\bf u}^{(m+1)}$ satisfies the inequality  $\eqref{4.20}_2$ for some lagre enough  $N_\de$.



Since  \eqref{4.17} holds also  for $ q = \infty$,  
we have only to derive the estimates for $|x|\geq cR\sqrt{t+ 1}$.

Let $\zeta$ be a smooth function with $\zeta=0$ in some neighborhood of $\partial \Om$.
According to Lemma \ref{integral_stokes0} the following representation holds for ${\bf u}^{(m+1)}$: 
\begin{align}
\label{integral_stokes2}
\begin{split}
{ \bf u}^{(m+1)}(x,t)\zeta(x)& =\int_\Omega {\bf u}_{0,i}(y)\zeta(y) \Gamma(x-y,t)dy\\
& \quad -\int_\Omega {\bf u}_0(y) \cdot [(\nabla\zeta(y))\partial_{y_i}\omega^i(x-y,t) 
 - (\nabla\zeta(y))\times \nabla\times \omega^i(x-y,t)]dy\\
  & + \int_\Omega { \bf u}^{(m+1)}(y,t)\cdot [(\nabla\zeta(y))\partial_{y_i}N(x-y) -(\nabla\zeta(y))\times \nabla\times (N(x-y){\mathbf e}^i)]dy\\
&\quad +\sum_{j=1}^2\int^{t}_0\int_\Omega  { \bf u}^{(m+1)}(y,\tau)\cdot R^i_j(x,y,t-\tau) dyd\tau\\
& \quad+ \int_0^t \int_\Om  
({\bf u}^{(m)}\otimes {\bf u}^{(m)})(y,\tau):\nabla  \Big(\zeta (y) {\bf G}^i (x-y, t-\tau)   \Big) dy d\tau\\
& \quad+ \int_0^t \int_\Om 
({\bf u}^{(m)}\otimes {\bf u}^{(m)})(y,\tau):\nabla \Big( \na \zeta (y) \times [\na \times \om^i (x-y, t-\tau)] \Big)   dy d\tau\\
&=I_1 +I_2 + J_3 + J_4+J_5 + J_6 + J_7.
\end{split}
\end{align}
Here $I_1, I_2$ are the same term appeared    in \eqref{integral_stokes_1}.

For the later use, we divide $J_6$ by
$J_6=J_{61}+J_{62}$, where
\[
J_{61}=\int^t_0\int_\Om { \bf u}^{(m)}\otimes{ \bf u}^{(m)}: \nabla \zeta(y) \cdot {\bf G}^i(x-y,t-\tau) dy d\tau,\]
\[
J_{62}=\int^t_0\int_\Om { \bf u}^{(m)}\otimes { \bf u}^{(m)}:  \zeta(y) \cdot \nabla {\bf G}^i(x-y,t-\tau) dy d\tau.\]

The following estimate will be used for the estimate of $J_{62}$ later on.
\begin{lemm}
\label{lemma10}
Let $A,B>0$, and  $b>n>a$. Then for $|x|\geq cA$ we have
\begin{align}
\label{lemma11} \int_{{\mathbb R}^n}(|y|+A)^{-a}(|x-y|+B)^{-b} dy\leq 
c|x|^{-a}B^{n-b}.
\end{align}
\end{lemm}
\begin{proof}
The estimate can be obtained by the straightforward computation after dividing the domain of integration
by $D_1=\{y:|y|\leq \frac{|x|}{2}\}$ and $D_2=\{y:|y|>\frac{|x|}{2}\}$.
We omit the details.

\end{proof}

{\bf $\bullet$  Estimate for $ 1\le  \al \leq n-1$.}

Let us  consider the case $ 1\le  \al \leq n-1$. 
 We take  $\zeta$ satisfying $\zeta(y)=0$ on $B_R$ and $\zeta = 1$ on $B_{2R}^c$. 
Note that $|\na^k \zeta (y) | \leq  c R^{-k}$ in $ y \in \Om_{2R} \setminus \Om_R$ and $\na^k \zeta (y) =0$ in $ (\Om_{2R} \setminus \Om_R)^c$   for $ k \in {\mathbb N}$.

According to the estimate  for  the case $\al\leq n-1$  in Step 2 of the section \ref{decomposition} 
\[
|I_1 (x,t)|\leq cM_0(|x|+\sqrt{t}+1)^{-\alpha}  \quad\mbox{ for } 3R\sqrt{t+1} < |x|
\]
and 
\begin{align*}
|I_2(x,t)|
\leq cM_0 (|x|+\sqrt{t} +1)^{-n+1}  \quad \mbox{ for } 3R\sqrt{t+1} < |x|.
\end{align*}

 Throughout this section we use  Lemma \ref{lemma6_1} for the estimates of $\nabla^k\omega$.
According to \eqref{4.17}
\[\|{\bf u}^{(m+1)}(t)\|_{L^\infty(\Om)}\leq c(1+t)^{-\frac{\al}{2}}M_0.\]
 By the same argument as for the estimate of $I_3-I_5$ for  the case $\al\leq n-1$  in Step 2 of section \ref{decomposition}  
we have 
\begin{align*}
|J_3(x,t)|
 \leq c M_0(|x|+\sqrt{t} +1)^{-n+1}\quad \mbox{ for } 3R\sqrt{t+1} < |x|,
\end{align*}
\begin{align*}
| J_4(x,t)| 
 \leq cM_0 ( |x| +\sqrt{t} +1)^{-\al}\quad \mbox{ for } 3R\sqrt{t+1} < |x|,
\end{align*}
\begin{align*}
|J_5(x,t)|
\leq c M_0(|x|+\sqrt{t} +1)^{-\frac{\al}{2}}\quad \mbox{ for } 3R\sqrt{t+1} < |x|.
\end{align*}


Again, according to \eqref{4.17}
\[\|{\bf u}^{(m)}(t)\|_{L^\infty(\Om)}\leq c(1+t)^{-\frac{\al}{2}}M_0.\]  
Hence we have
\begin{align*}
J_{61}+J_7
&\leq c\int^t_0\Big((|x|+\sqrt{t-\tau})^{-n}+(|x|+\sqrt{t-\tau})^{-n+1}\Big)\|{\bf u}^{(m)} (\tau)\|_{L^\infty (\Om_R)}^2 d\tau\\
&\leq cM^2_0(|x|^{-n}+|x|^{-n+1}) 
       \int^t_0(1+\tau)^{-\al} d\tau\\
& \leq c_{14}M_0^2(|x|+\sqrt{t}+1)^{-\al}\mbox{ for }R\sqrt{t+1} <|x|.
\end{align*}

Use \eqref{4.17} and $\eqref{4.20}_1$ for  ${\bf u}^{(m)}$ 
and use  Lemma \ref{lemma10} for the estimate of $\int_\Om (|y|+\sqrt{\tau}+1)^{ -\al  }(|x-y|+\sqrt{t-\tau})^{-n-1}dy$. Then   we have 
\begin{align*}
J_{62}&\leq c\int^t_0\|{\mathbf u}^{(m)}(\tau)\|_{L^\infty(\Om)}\int_\Om |{\bf u}^{(m)}(y,\tau)|(|x-y|+\sqrt{t-\tau})^{-n-1}dyd\tau\\
&\leq cM_0N\int^t_0(1+\sqrt{\tau})^{-\al}\int_\Om (|y|+\sqrt{\tau}+1)^{ -\al  }(|x-y|+\sqrt{t-\tau})^{-n-1}dyd\tau\\
&\leq 
cM_0N  \int^t_0(1+\sqrt{\tau})^{-\al} (|x|+\sqrt{\tau}+1)^{-\al}(t-\tau)^{-\frac{1}{2}}d\tau\\
& \leq c_{15}M_0N (|x|+\sqrt{t}+1)^{-\al}\mbox{ for }R\sqrt{t+1} < |x|.
\end{align*}


Combining all the estimates, we obtain
\begin{equation*}
|{\bf u}^{(m+1)}(x,t)|\leq \big(  C_3M_0 +C_4M_0^2+C_5M_0N \big)   (|x|+\sqrt{t} +1 )^{-\al}  \quad 
\end{equation*}
for $\al\leq n-1$. 
Choose $N=3C_3M_0$ and $M_0$  so small that $C_4M_0\leq C_3, \ C_5M_0\leq \frac{1}{3}$, then
\begin{equation}
|{\bf u }^{(m+1)}(x,t)|\leq N (|x|+\sqrt{t} +1 )^{-\al}.
\end{equation}

{\bf  $\bullet$ Estimate for $ n-1< \al \leq  n$.}

Now, we will derive the pointwise estimate for the case  $ n-1 < \al \leq  n$.
 
 We fix $x \in B_{3R}^c$. 
We take  $\zeta_x = \zeta \in C^\infty_0 (B(x, \frac{|x|}{2}))$ satisfying $\zeta(y)=1$ on $B(x, \frac{ |x|}{4})$. 
Note that $|\na^k \zeta (y) | \leq  c |x|^{-k}$ in $ y \in  B(x, \frac{|x|}{2}) \setminus B(x, \frac{ |x|}{4}) $ and $\na^k \zeta (y) =0$ in $ ( B(x, \frac{|x|}{2}) \setminus B(x, \frac{ |x|}{4}))^c$   for $ k \in {\mathbb N}$.

%
%
%
%
%
%

According to the estimate  for  the case $n-1<\al\leq n$  in Step 2 of the section \ref{decomposition} 
\[
|I_1(x,t)|\leq cM_0(|x|+\sqrt{t}+1)^{-\alpha}\quad \mbox{ for }3R\sqrt{t+1} < |x|,
\]
%
\begin{align*}
|I_2(x,t)|
\leq cM_0(|x|+\sqrt{t}+1)^{-\al } \quad \mbox{ for }3R\sqrt{t+1} < |x|.
\end{align*}

Let $\de > 0$ be small. Take $q_\de$ with $ \frac{n}\al < q_\de = \frac{n}{\al -\de}$. 
 Recall that 
 ${\bf u}^{(m+1)}$  satisfies \eqref{4.17} for any $\frac{n}{\al}<q\leq \frac{n}{\al-1}$. 
 By the same argument as for the estimate of $I_3-I_5$ for  the case $n-1<\al\leq n$  in Step 2 of section \ref{decomposition}  
we have 
\begin{align*}
|J_3(x,t)|
\leq c_\delta M_0(|x|+\sqrt{t}+1)^{-\al+\de}\quad \mbox{ for }   3R\sqrt{t+1} < |x|,
\end{align*}
\begin{align*}
|J_4(x,t)| 
 \leq c_\delta M_0  (  |x| +\sqrt{t} +1)^{-\al +\delta } \quad \mbox{ for }  3R\sqrt{t+1} <|x|,
\end{align*}
%
\begin{align*}
|J_5(x,t)| 
\leq c_\delta  M_0  (  |x| +\sqrt{t} +1)^{-\al + \delta } \quad \mbox{ for } 3R\sqrt{t+1} <|x|.
\end{align*}

Recalling ${\bf u}^{(m)}$ satisfies \eqref{4.17} for  any $\frac{n}{\al}<q\leq \infty$,  we have
%
\begin{align*}
J_{61}+J_7&\leq c\int^t_0 \int_{ B(x, \frac{|x|}{2}) \setminus B(x, \frac{ |x|}{4}) } \Big(  |\nabla \zeta(y)||\nabla^2\omega(x-y,t-\tau)|\\
&\qquad\qquad+|\nabla^2\zeta(y)||\nabla\omega(x-y,t-\tau)|   \Big)|{\bf u}^{(m)} (y,\tau)|^2 dy d\tau\\
&\leq c|x|^{-n-1}\int^t_0\int_{ B(x, \frac{|x|}{2}) \setminus B(x, \frac{ |x|}{4})}   |{\bf u}^{(m)} (y,\tau) |^2dyd\tau\\
&  \leq c |x|^{-\frac{n}{q_\de}-1} \int^t_0    \|{\bf u}^{(m)} (\tau)\| _{L^{\infty} (\Om)} \|{\bf u}^{(m)} (\tau)\| _{L^{q_\de} (\Om)} d\tau\\
&  \leq cM_0^2   |x|^{-\frac{n}{q_\de} -1}  \int^{t}_0    (1 +\tau)^{-\al+\frac{n}{2q_\de}}     \ln^{2\de_{\al n}} {(2+\tau)} 
                       d\tau\\
& \leq c M_0^2 |x|^{-\frac{n}{q_\de} -1}\int^t_0(1+\tau)^{-\frac{\al}{2}}d\tau \\
&\leq c_\de M_0^2(|x|+\sqrt{t}+1)^{-\al+\delta }\mbox{ for } R\sqrt{t+1} < |x|. 
\end{align*}


Note that
\[
\sup_{x\in \Om}|x|^{n-1}|{\bf u}_0(x)|\leq \sup_{x\in \Om}|x|^{\al}|{\bf u}_0(x)|\mbox{ for }\al>n-1.\]
Hence if $n-1<\al\leq n$, then 
according to the result of the previous step, 
\[
|{\bf u}^{(k)}(x,t)|\leq c M_0(|x|+\sqrt{t}+1)^{-n+1}\mbox{ for all }k=0,1,2,\cdots.
\]
Since $\mbox{\rm supp }\zeta\subset B(x,\frac{|x|}{2})$ we have \begin{align*}
J_{62}&\leq c\int^t_0 \int_{B(x,\frac{|x|}{2})}
|{\bf u}^{(m)}(y,\tau)|^{2}(|x-y|+\sqrt{t-\tau})^{-n-1}dyd\tau\\
&\leq cM_0^2\int^t_0\int_{B(x,\frac{|x|}{2})}  (|y|+\sqrt{\tau}+1)^{-2n+2}  (|x-y|+\sqrt{t-\tau})^{-n-1}     dyd\tau\\
&\leq 
c_\de M_0^2 |x|^{-2n+2}\int^t_0\int_{B(x,\frac{|x|}{2})}   (|x-y|+\sqrt{t-\tau})^{-n-1}     dy  d\tau\leq cM_0^2|x|^{-2n+2}\int^t_0(t-\tau)^{-\frac{1}{2}}d\tau\\
&\leq cM_0^2|x|^{-2n+2}t^{\frac{1}{2}}\leq c_{\de}M_0^2 (|x|+\sqrt{t+1})^{-\al+\de} \mbox{ for  }
  R\sqrt{t+1} < |x|\mbox{ and }n\geq 3.
\end{align*}


Combining all the estimtes, we obtain
\begin{equation*}
|{\bf u }^{(m+1)}(x,t)|\leq \Big(C_{6,\delta}M_0 +C_{7,\delta}M_0^2 \Big) (|x|+\sqrt{t} +1 )^{-\al+\de}  \quad 
\end{equation*}
for $n-1<\al\leq n. $ 
Choose $N_\de$ so large that $N_\de>C_{6,\delta}M_0 +C_{7,\delta}M_0^2 $,  then
\begin{equation}
\label{equati100}
|{\bf u}^{(m+1)}(x,t)|\leq  N_\de (|x|+\sqrt{t} +1 )^{-\al+\de}  \quad 
\end{equation}
 We complete the proof of  Corollary \eqref{navier_wieght_space}.

\appendix
\numberwithin{equation}{section}
\setcounter{equation}{0}

\section{Proof of Lemma \ref{lemma1}}
\setcounter{equation}{0}
\label{appendix.lemma1}

Decomposing  the domain of integration  into three subdomains such as  
\begin{align}
 D_1=\{y\,| \,|x-y|\leq \frac{|x|}{2}\}, \quad  D_2=\{y\, |\, |y|\leq \frac{|x|}{2}\}, \quad D_3=\{y \,| \, |x-y|\geq \frac{|x|}{2} , |y|>\frac{|x|}{2}\}.
\end{align}
Let 
\begin{align}
I=\int_{D_1}\Gamma(x-y,t)v_0(y)dy, \quad   II=\int_{D_2}\Gamma(x-y,t)v_0(y)dy, \quad  III=\int_{D_3}\Gamma(x-y,t)v_0(y)dy
\end{align}
so that ${\bf V}=I+II+III$. 

Observe that  $\frac{|x|}{2}\leq  |y|\leq \frac{3|x|}{2}$ on $D_1$. Then 
 \begin{align*}
 I& \leq cM_0 (|x|+1)^{-\alpha}\int_{|x-y|\leq \frac{|x|}{2}}\Gamma(x-y,t)dy \leq  \left\{\begin{array}{l} \vspace{2mm}
cM_0 |x|^{-\alpha}t^{-\frac{n}{2}}|x|^{n}  \mbox{ if }  |x|\leq \sqrt{t}\\
cM_0 |x|^{-\alpha} \quad\mbox{ if }|x|\geq \sqrt{t}
\end{array}\right.\leq  \frac{ c M_0}{(|x|+\sqrt{t})^\alpha}.
 \end{align*}

%
 
  Since $ |x-y|\geq \frac{|x|}{2}$ on $D_2$, we have 
 \begin{align*}
 II& \leq cM_0t^{-\frac{n}{2}} \int_{|y|\leq\frac{|x|}{2}}e^{-\frac{|x|^2}{16t}}(|y|+1)^{-\alpha} dy\\
 &  \leq cM_0t^{-\frac{n}{2}}  e^{-\frac{|x|^2}{16t}}\int_{|y|\leq\frac{|x|}{2}}(|y|+1)^{-\alpha} dy\\
 & \leq \left\{ \begin{array}{l} \vspace{2mm}
 cM_0t^{-\frac{n}{2}}  e^{-\frac{|x|^2}{16t}}|x|^{n-\alpha}\leq c M_0\frac{1}{(|x|+\sqrt{t})^\alpha},\quad \alpha<n\\
cM_0t^{-\frac{n}{2}}  e^{-\frac{|x|^2}{16t}}\ln{(1+|x|)} \leq c M_0\frac{ \ln{(1+|x|)}}{(|x|+\sqrt{t})^n}, \quad \alpha=n .
 \end{array}\right.
 \end{align*}
 
  Since $ |x-y|\geq \frac{|y|}{3}$ on $D_3$, we have 
 \begin{align*}
IIII& \leq cM_0t^{-\frac{n}{2}} \int_{|y|\geq \frac{|x|}{2}}e^{-\frac{|y|^2}{16t}}(|y|+1)^{-\alpha} dy \\
 & \leq cM_0t^{-\frac{n}{2}}  e^{-\frac{|x|^2}{32t}}\int_{|y|\geq \frac{|x|}{2}}e^{-\frac{|y|^2}{32t}}|y|^{-\alpha} dy\\
 & = cM_0t^{-\frac{\alpha}{2}}  e^{-\frac{|x|^2}{32t}}\int_{|\eta|\geq \frac{2|x|}{\sqrt{t}}}e^{-\frac{|\eta|^2}{32}}|\eta|^{-\alpha}d\eta\\
 & \leq cM_0t^{-\frac{\alpha}{2}}  e^{-\frac{|x|^2}{32t}}\leq \frac{cM_0}{(|x|+\sqrt{t})^\alpha}.
 \end{align*}

Therefore we obtain
\begin{equation}
\label{eq2}
|V(x,t)|\leq \left\{ \begin{array}{l} \vspace{2mm}
  cM_0\frac{1}{(|x|+\sqrt{t})^\alpha},\quad \alpha<n\\
 cM_0\frac{ \ln{(1+|x|)}}{(|x|+\sqrt{t})^n}, \quad \alpha=n.
 \end{array}\right..
\end{equation}

On the other hand, by Young's theorem we have
\begin{equation}
\label{eq1}
\|{ V}(t)\|_{L^\infty({\mathbb R}^n)}\leq c\|v_0\|_{L^\infty({\mathbb R}^n)}.
\end{equation}
Combining \eqref{eq1} and \eqref{eq2} we conclude that
\begin{align}
\label{eq3} 
{ V}(x,t)\leq c 
M_0(1+|x|+\sqrt{t})^{-\alpha}   ( \ln {(2+ | x|)} )^{\de_{\al n}},\quad  0 < \alpha\leq n.
 \end{align}

Now,  let $\frac{n}{\alpha}<q<\infty$ for $\alpha<n$ and $1<q<\infty$ for $\alpha=n$.

By the change of variable 
\[
\Big(\int_{{\mathbb R}^n}(|x|+1+\sqrt{t})^{-q\alpha} dx\Big)^{\frac{1}{q}}=c(1+\sqrt{t})^{-\al+\frac{n}{q}}\Big(\int_{{\mathbb R}^n}(|\eta|+1)^{-q\alpha}d\eta\Big)^{\frac{1}{q}}
,\]
and
\begin{align*}
&\Big(\int_{\R} (1 +\sqrt{t} +|x|)^{-nq}\ln^q{(1+\sqrt{t} +|x| )}  dx\Big)^{\frac{1}{q}}\\
&= c (1 +\sqrt{ t})^{-n + \frac{n}{q}} \Big(\int_{\R} (1 +|x| )^{-nq} \big( \ln^q{(1+(1+\sqrt{t})|x|)} \big) dx\Big)^{\frac{1}{q}}\\
& \leq  c (1 + t)^{-\frac{n}2  + \frac{n}{2q}} \Big(\int_{\R} (1 +|x| )^{-nq} \big( \ln^q{(2 +\sqrt{t}))}   + \ln^q{(1+|x|)} \big) dx\Big)^{\frac{1}{q}}
.
\end{align*}

Observe 
\[
\Big(\int_{{\mathbb R}^n}(|\eta|+1)^{-q\alpha} dx\Big)^{\frac{1}{q}}\leq c\quad \frac{n}{\alpha}<q<\infty.\]
and
\[
\Big(\int_{{\mathbb R}^n}(|\eta|+1)^{-qn}\ln^q{(1+|\eta|)} dx\Big)^{\frac{1}{q}} \leq  c\quad  1<q<\infty.
\]


Combining the all the estimates we obtain
%
\[
\|{V}(t)\|_{L^q({\mathbb R}^n)}\leq cM_0  
(1+t)^{-\frac{\alpha}{2}+\frac{n}{2q}}   \ln^{\de_{\al n}} {(2+t)}, \quad \frac{n}{\alpha}<q<\infty.
\]

\end{document}